\documentclass[11pt]{amsart}
\usepackage[top=1in, bottom=1in, left=1.3in, right=1.3in]{geometry}
\usepackage{changes}

\usepackage{amsmath}
\usepackage{amssymb}
\usepackage{amsthm}
\usepackage{verbatim}
\usepackage{graphicx}
\usepackage{graphics}
\usepackage{color}
\usepackage{enumerate}
\usepackage{mathrsfs}
\usepackage{booktabs}
\usepackage[toc,page]{appendix}
\usepackage{enumerate}
\usepackage{mathrsfs}
\usepackage{fancyhdr}
\usepackage{amsopn, amstext, amscd,pifont}
\usepackage{amssymb,bm, cite}
\usepackage[all]{xy}
\usepackage{color}
\usepackage{graphicx,hyperref}
\usepackage{float}
\usepackage{listings}
\usepackage{setspace}
\usepackage{url}
\usepackage{paralist}
\usepackage{subcaption}
\usepackage{upgreek}

\usepackage{graphicx,pgf}
\usepackage{float}
\usepackage{enumitem}

\usepackage{fancyhdr}
\numberwithin{equation}{section} 

\newtheorem{introtheorem}{Theorem}

\theoremstyle{plain}

\newtheorem{theorem}{Theorem}[section]
\newtheorem{proposition}[theorem]{Proposition}
\newtheorem{lemma}[theorem]{Lemma}
\newtheorem{corollary}[theorem]{Corollary}

\theoremstyle{definition}

\newtheorem{remark}[theorem]{Remark}

\graphicspath{{images/}}

\def\C{\mathbb{C}}

\def\R{\mathbb{R}}

\def\P{\mathbb{P}}
\def\Q{\mathbb{Q}}
\def\N{\mathbb{N}}
\def\L{\mathbb{L}}
\def\Z{\mathbb{Z}}
\def\D{\mathbb{D}}

\def\QS{\Q / \Z}

\def\CP2{{\mathbb{CP}^2}}

\usepackage{calrsfs}
\DeclareMathAlphabet{\pazocal}{OMS}{zplm}{m}{n}
\def\ccS{{\mathcal{S}}}

\def\ccS{{\mathcal{S}}}

\def\cE{{\pazocal{E}}}

\def\cG{{\pazocal{G}}}
\def\cH{{\pazocal{H}}}

\def\cK{{\pazocal{K}}}
\def\cL{{\pazocal{L}}}
\def\cM{{\pazocal{M}}}

\def\cS{{\pazocal{S}}}

\def\cX{{\pazocal{X}}}

\def\oC{\overline{\C}}

\def\omegap{{\omega^\prime}}

\def\tvarphi{\tilde{\varphi}}

\def\trz{\tilde{z}}

\def\poneberk{\mathbb{P}^{1, {an}}}
\def\pone{\mathbb{P}^1}
\def\diam{\mathrm{diam}}

\newcommand{\val}[1]{\ensuremath{\left| #1 \right|_{\mathrm{o}}}}
\newcommand{\nor}[1]{\ensuremath{\left\| #1 \right\|}}

\def\per{\operatorname{Per}}

\def\ocS{\overline{\cS}}

\def\ord{\operatorname{ord}}
\def\dist{\operatorname{dist}}
\def\rd{\operatorname{red}}

\def\tgamma{{\tilde{\gamma}}}

\def\laurent{\C((t))}
\def\laurentm{\C((t^{1/m}))}
\def\puiseux{\C \langle\langle t \rangle\rangle }
\def\ponel{\P^1_\L}
\def\ponec{\P^1_\C}

\def\lval{\left|}
\def\rval{\right|_\mathrm{o}}

\begin{document}

\title[Euler Characteristic of Periodic Curves]{Asymptotics of Transversality
  in Periodic Curves of
  Quadratic Rational Maps}
\author{Jan Kiwi}
\thanks{The author acknowledges the support of ANID/FONDECYT Regular
1240508.} 
\address{Facultad de Matem\'aticas, Pontificia Universidad Cat\'olica de Chile} 
\email{jkiwi@uc.cl}

\maketitle
\begin{abstract}
  We compute the Euler characteristic of the moduli space of quadratic
  rational maps with a periodic marked critical point of a given period. 
\end{abstract}

\section{Introduction}

The dynamical properties of a rational map $f: \pone_\C \to \pone_\C$ are determined, to a large extent, by its critical orbits.
Therefore, parameter spaces are frequently explored by studying the
loci of maps  with a prescribed critical orbit behavior.
Here we focus on quadratic rational maps.
Specifically, our interest is on the moduli spaces $\cS_p$  of quadratic maps with a marked periodic critical point of period, for any $p \ge 1$.
These moduli spaces  are punctured Riemann surfaces, often called
\emph{periodic curves}.
Periodic curves and related spaces, including critical orbit relation curves  (COR curves), have received considerable
attention in the literature (see e.g. 
\cite{buff2022prefixed,demarco2015bifurcation,FirsovaKahnSelinger,HironakaKoch,MilnorQuadratic,RamadasSilversmithEquations,RamadasGleason,Stimson,ReesII,ReesAsterisque}). 
However, for general $p$, basic questions about the global topology of $\cS_p$ remain open: Is $\cS_p$ connected? If so,  what is the genus $g_p$ of $\cS_p$? 
What is the number of punctures $N_p$ of $\cS_p$?
Our main result is a formula for the Euler characteristic of $\cS_p$. In particular, provided that $\cS_p$ is connected, this establishes the equivalence 
between computing $g_p$ and computing $N_p$.
Our result arises from a precise control of the dynamical and parameter space
scales near the punctures, with the aid of non-Archimedean dynamics.

In order to state the main result, we need to introduce two numbers: $\eta'(p)$
and $\eta'_{II}(p)$. The first one, $\eta'(p)$,  is the number of parameters in the quadratic family $Q_c(z):=z^2 + c$ such that the critical point $z=0$  is periodic of exact period $p$. The second one, $\eta'_{II}(p)$, is the number
of parameters in $\cS_p$ such that the two critical points lie in the same
periodic orbit; that is, the number of \emph{type II or bitransitive centers} in $\cS_p$.
A formula for $\eta'(p)$ follows from well known results about the quadratic family, due to
Douady, Hubbard and, Gleason~\cite{OrsayNotes}. In~\cite{KiwiRees}, with M. Rees, we obtained a formula for $\eta'_{II}(p)$. These formulas are discussed in
\S \ref{counting-polynomial-ss} and \S \ref{counting-rational-ss}.

\begin{introtheorem}
  \label{A}
  Let $\chi(\cS_p)$ denote the Euler characteristic of $\cS_p$.
  Then, for all $p \ge 3$, 
  $$\chi(\cS_p) = \dfrac{ 2 \eta'(p)}{3} - \eta'_{II} (p).$$ 
\end{introtheorem}
The curves $\cS_1$ and $\cS_2$ are isomorphic to $\C$ and $\C^* =\C\setminus \{0\}$, respectively.

Denote by $\widehat{\cS_p}$ the smooth compactification of $\cS_p$ obtained by
filling each puncture with an ideal point.  Then, for $p \ge 3$,
$$\chi(\widehat{\cS_p}) = N_p + \dfrac{ 2 \eta'(p)}{3} - \eta'_{II} (p).$$
As mentioned above, we lack of formulas and general methods to compute $N_p$.

There is a strong analogy between the study of cubic polynomials and
quadratic rational maps. In both cases, two critical points are relevant
in phase space and moduli spaces are complex two-dimensional.
However, substantially more is known about cubic polynomials.
In fact, with Arfeux, we proved that the moduli space $\ccS_p$ of cubic polynomials
with a marked periodic critical point, of period $p$, is
connected~\cite{arfeux2023irreducibility}. With Bonifant and Milnor, we obtained a formula for the Euler characteristic of $\ccS_p$ in~\cite{AKMCubic}. Schiff and De Marco~\cite{DeMarcoSchiff} produced an algorithm to determine the number of punctures of $\ccS_p$, but no general formula is known.

The strategy to compute the Euler characteristic of $\cS_p$ is similar to the one employed for cubic polynomials in~\cite{AKMCubic}.
It consists on computing the degree of a naturally defined meromorphic one-form $d \tau$ on $\widehat{\cS_p}$ with all its zeros and poles located at the punctures of $\cS_p$.
Near the punctures, elements of $\cS_p$ correspond to conjugacy classes
of maps close to infinity in the moduli space of all quadratic rational maps. 
Our computation of the order of $d \tau$ at the punctures of $\cS_p$
relies on the study of degenerate rational maps with tools from
non-Archimedean dynamics. We reinterpret the formulas for
the order of $d \tau$ at the punctures, obtained in \cite{AKMCubic}.
Our contribution is to show that the underlying principle is that the order at the punctures
of a  parameter space derivative 
and a dynamical space derivative agree,
modulo a minor correction factor (see Lemma~\ref{l-main}).
This may be interpreted as an equation for the asymptotics (at the punctures)
of the transversality responsible for the smoothness of $\cS_p$. 
It would be interesting to investigate
the extent up to which similar equations for the \emph{asymptotics of transversality} hold for general COR curves. 

The paper is organized as follows:

In \S \ref{s:preliminaries}, we introduce the basic notation and background
results regarding quadratic rational maps and the curves $\cS_p$.

In \S \ref{s:asymptotics}, we introduce the meromorphic one form used
to compute the Euler characteristic of $\cS_p$. 
Our Main Lemma \ref{l-main} is stated. However, its proof is postponed until \S \ref{s:puiseux}.
Assuming the Main Lemma, we prove Theorem~\ref{A}.

In \S \ref{s:parabolic}, we summarize some results
regarding the punctures of $\cS_p$. The
emphasis is on parabolic rescaling limits,
following the works of Stimson, Epstein, Rees, and De Marco~\cite{Stimson,EpsteinBounded,ReesAsterisque,DeMarcoQuadratic}.

In \S \ref{s:puiseux}, we prove the Main Lemma. We start translating
it into a lemma regarding quadratic rational dynamics over
the field of Puiseux series. To prove it,
we abuse of the detailed knowledge
of quadratic rational dynamics over this field obtained in \cite{KiwiPuiseuxQuadratic}. Dynamics over non-Archimedean fields is better understood
in the Berkovich projective line. However, since the Main Lemma involves
only quantities related to the (classical) projective line, for simplicity,
we have chosen to do most of our work without reference to Berkovich spaces.
Nevertheless, at some points, we have to translate results from
\cite{KiwiPuiseuxQuadratic}, written in the language of Berkovich spaces,
to our context. 
As a byproduct of the proof of the Main Lemma,
we obtain  Mandelbrot tori embedded 
in quadratic rational moduli space 
reminiscent of the ones obtained by Branner and Hubbard \cite{BrannerHubbardCubicI} for cubic polynomials
(see Corollary~\ref{c:mandelbrot}).

{\subsection*{Acknowledgments}
  Part of this work was done when the author was visiting  the Institute for Mathematical Sciences (IMS) at Stony Brook University and
  the Institut de Math\'ematiques de Toulouse (IMT). The author
  is grateful to the IMS and IMT for their hospitality.

\section{Preliminaries}
\label{s:preliminaries}

\subsection{Moduli space}

Denote the moduli space of quadratic rational maps with a marked critical point by $\cM_2^{cm}$. More precisely, consider dynamical pairs $(f,\omega)$ where $f: \pone_\C \to \pone_\C$ is a quadratic rational map and $f'(\omega)=0$. We say that  two such pairs $(f,\omega)$ and  $(g,\omega')$ are conjugate if there exists a M\"oebius transformation $\gamma$ such that $\gamma \circ g =  f \circ \gamma$ and $\gamma(\omega') = \omega$.
According to Lemma~6.1 in \cite{MilnorQuadratic}, the moduli space of critically marked quadratic rational maps $\cM_2^{cm}$ is an algebraic surface with a unique singular point at the conjugacy class of  $(z \mapsto z^{-2},\omega=0)$.
We refer to the elements of $\cM^{cm}_2$ simply as maps.
Also, $\cS_p$ is an algebraic subvariety of $\cM_2^{cm}$, for all $p\ge 1$.
The curves $\cS_p$ are known to be smoothly embedded in $\cM_2^{cm}$
with the sole exception of $\cS_2$ at the singular point of $\cM_2^{cm}$ (e.g. see~\cite{LWVcor,EpsteinTransversality}c.f.~\cite[Theorem~2.5]{KiwiRees}). Conjecturally $\cS_p$ is irreducible for all $p\ge 1$ (c.f. \cite{RamadasGleason}).

\subsection{Counting in the Quadratic Family}
\label{counting-polynomial-ss}
Abstractly, for $p \ge 1$, the numbers $\eta'(p) \in \N$ are defined by:
$$2^{p-1} = \sum_{d|p} \eta'(d).$$
Therefore, $$\eta'(p) = \dfrac{1}{2} \sum_{d|p} \mu(d)\cdot 2^{p/d},$$ 
where $\mu$ is the M\"obius function.

It is however convenient to relate $\eta'(p)$ with quadratic polynomial dynamics, as in the introduction. Recall that $Q_c(z) := z^2 +c$ denotes the quadratic family.
Then, 
\begin{equation}
  \label{eq:eta}
  \eta'(p) = |\{ c \in \C : z=0 \text{ has period } p \text{ under } Q_c \}|.
\end{equation}
In fact, according to Gleason (e.g. see~\cite[Expos\'e XIX]{OrsayNotes}), all the roots of
the degree $2^{p-1}$ polynomial $Q_c^{\circ p-1}(c)$ are simple.
Thus, both sides of \eqref{eq:eta} must be equal.

\bigskip
In order to write a formula for $\eta'_{II}(p)$, we need to define $\nu_q(k)$, for
$1<k\le q$.
Let $0 \le r < q$ be such that $p = r \mod q$.
Then
$$\nu_q (k) :=
\begin{cases}
  \dfrac{1}{2} \cdot \dfrac{2^k-2^r}{2^q-1} & \text{ if } r \neq 0, \\
  &\\
  \dfrac{1}{2} + \dfrac{1}{2} \cdot \dfrac{2^k-1}{2^q-1}& \text{ if } r = 0.
\end{cases}
$$
For example,

\begin{equation}
  \label{eq:nu2}
\nu_2(p) := \begin{cases}
        \dfrac{2^n-2}6 & \text{ if } n \text{ is odd,}\\ \\
        \dfrac{2^n - 4}6 & \text{ if } n \text{ is even.}
      \end{cases}
    \end{equation}

      The numbers $\nu_q(k)$ are also closely related to the quadratic family.
      More precisely, $\nu_q(k)$ is the number of parameters $c$ in
      the $1/q$-limb of the Mandelbrot set 
      for which the critical point
      $c=0$ is periodic of period, at least $3$, dividing $k$ (e.g. see~\cite{MilnorPOP} for the definition of limbs).

      In order to employ some of the results in~\cite{KiwiRees},
      our next lemma records the simple relation between
      the numbers $\nu_2(p)$ and $\eta'(p)$:

      \begin{lemma}
        \label{etap-l}
        Let $\nu_2'(p)$ be such that $$\nu_2 (p) = \sum_{3 \le d |p} \nu_2'(d).$$
        Then, $\eta'(p) = 3 \nu'_2(p)$.
      \end{lemma}

      \begin{proof}
        It is sufficient to prove that, for all $p \ge 3$,
        we have $$\sum_{3 \le d |p} \eta'(d) = 3 \nu_2(p).$$
        Indeed, from \eqref{eq:nu2}, for $p$ even,
        $$ \sum_{3 \le d |p} \eta'(d) = 2^{p-1} - \eta'(1) - \eta'(2) = 2^{p-1}-2 = 3 \nu_2(p).$$
        A similar identity holds for $p$ odd. 
      \end{proof}

      \subsection{Counting in $\cS_p$}
      \label{counting-rational-ss}
A map $f \in \cM_2^{cm}$ is \emph{hyperbolic} if both of its critical orbits are in the basin(s) of some attracting cycle(s). A connected component of hyperbolic maps is called a \emph{hyperbolic component}.
Following Rees~\cite{rees1990components}, hyperbolic components are classified into four types according to the relative position of the critical points in the attracting basins. Here we are mainly concerned with type II components, also known as \emph{bitransitive or type B components}.
A \emph{type II hyperbolic component} consists of maps having a unique attracting cycle of period $p \ge 2$ such that both critical points belong
to periodic Fatou components. These  Fatou components are necessarily distinct. Each type II hyperbolic component $\cH$  of period $p$ contains a unique postcritically finite map $f_0$ called the \emph{center} of $\cH$.
The critical points of $f_0$ lie in the same orbit of period $p$. In particular, $f_0 \in \cS_p$.

In~\cite{KiwiRees} we obtained a formula for the number $\eta_{II}'(p)$ of type II hyperbolic components in $\cM_2^{cm}$ of exact period $p$. It is not difficult to check that the
formula written below is equivalent to the one obtained in ~\cite[Theorem 1.2]{KiwiRees}.

\begin{theorem}
  For $p \ge 3$, $$ \sum_{3 \le d|p} \dfrac{\eta'_{II} (d)}{d}= \dfrac{\nu_2(p)}{2} - \dfrac{A(p)}{2p},$$
  where
  $$A(p) := p \cdot \sum_{q=2}^{p-1} \phi(q) \left( \sum_{k=1}^{p-1} \nu_q(k) (2^{p-k-1}- \nu_q(p-k))\right),$$
  and $\phi$ denotes the Euler totient function.
\end{theorem}

The M\"oebius inversion formula yields:
$$\eta'_{II} (p) = \dfrac{1}{2} \left( p \nu'_2(p) + \sum_{d|p, p/d \ge 3} d \cdot \mu(d) \cdot A \left(\dfrac{p}{d}\right) \right).$$

\subsection{Partial parametrization of  $\cM_2^{cm}$}
The complement of
$\cS_1 \cup \cS_2$ in  $\cM_2^{cm}$
is identified with $\C^* \times \C$
via the family 
$$f_{a,b}(z) = 1 + \dfrac{b}{z} + \dfrac{a}{z^2}$$
with marked critical point $\omega =0$. That is,
$$\cM_2^{cm} \setminus (\cS_1 \cup \cS_2) = \{ [(f_{a,b},0)] : (a,b) \in 
\C^* \times \C\}.$$
In fact, observe that if $[(f,\omega)] \in \cM_2^{cm} \setminus (\cS_1 \cup \cS_2)$, then $\omega, f(\omega)$ and $f^{\circ 2}(\omega)$ are distinct. Normalizing so that
$\omega =0$, $f(\omega)=\infty$ and $f^{\circ 2}(\omega)=1$, we have that $f=f_{a,b}$, for some $(a,b) \in \C^* \times \C$.

Regard the projective plane $\P^2_\C$ as
a compactification of $\cM_2^{cm} \setminus (\cS_1 \cup \cS_2)$
via the inclusion $\C^* \times \C \hookrightarrow \P^2_\C$ given by $(a,b) \mapsto [a:b:1]$. Thus, $\cM_2^{cm} \setminus (\cS_1 \cup \cS_2)$ is
identified with the complement, in $\P^2_\C$,  of the lines $$\cL_+ := \{ [a:b:c] : a=0\}$$
and $$\cL_- := \{ [a:b:c]: c = 0 \}.$$
From now on we regard $\cS_p$, for any $p \ge 3$, as a subset of $\P^2_\C$ and
denote its closure by $\ocS_p$.

Milnor~\cite{MilnorQuadratic} introduced a compactification of the moduli space
of quadratic rational maps, without critical markings, by a line at infinity
corresponding to pairs $\{\lambda,\lambda^{-1}\}$ where $\lambda \in \oC$.
Points in moduli space close to  $\{\lambda,\lambda^{-1}\}$  have
two fixed points with multipliers close to $\lambda$ and $\lambda^{-1}$.
Silverman~\cite{SilvermanModuli}
showed that, in fact, Milnor's compactification is
the GIT-compactification of unmarked moduli space.
Maps close to  $\cL_-$ and $\cL_+$ in marked moduli space,
after forgetting the marking, become maps in unmarked moduli space close to infinity.
The line $\cL_-$ should be regarded as a blow up of $\{-1,-1\}$.
More precisely,
as $[a:b:1]$ approaches $[1:m:0] \in \cL_-$ with $m \neq 0$, the multipliers of
two fixed points of $f_{a,b}$ converge to $-1$. Also, as
$[a:b:1]$ approaches $[0:m:1] \in \cL_+$ with $m\neq 0$, the multipliers of
two fixed points of $f_{a,b}$ converge to reciprocal complex  numbers
in $\C \setminus \{0,-1\}$.

\subsection{Degree of $\cS_p$}


      \begin{lemma}[{\cite[Lemma 3.1]{KiwiRees}}]
        \label{l-degree}
        For all $p \ge 3$,
$$\deg \ocS_p = \nu'_2(p).$$
  Moreover, none of the points $[1:0:0],[0:1:0],[0:0:1]$ belongs to $\ocS_p$.
\end{lemma}
Lemma~\ref{etap-l} yields:
\begin{equation}
  \label{eq:degree}
  \deg \ocS_p = \dfrac{\eta'(p)}{3}.
\end{equation}

We remark that~\cite[Lemma~3.1]{KiwiRees} states an equivalent
assertion but for 
$\overline{\cX_n}:=\cup_{3 \le p|n} \overline{\cS_p}$ instead of $\ocS_p$.

\subsection{Punctures}
Recall that the smooth compactification $\widehat{\cS_p}$ of $\cS_p$ is obtained from $\cS_p$ by filling each of its punctures with an ideal point.
Equivalently, it is the normalization of the projective curve $\ocS_p$.
The punctures of
$\cS_p$ correspond to germs of branches of $\ocS_p$ at the lines
$\cL_+$ and $\cL_-$. A representative of the branch corresponding to a puncture $x$ will be denoted by $\cE_x$. A punctured neighborhood
of $x$ in  $\widehat{\cS_p}$ is  $\cE^*_x := \cE_x \setminus \cL_\pm$.

Let $\pi: \widehat{\cS_p} \to \ocS_p$ be the natural projection. It is
a bijection in $\cS_p$ and maps each puncture ${x}$ onto the intersection
$\pi(x)$ of the corresponding branch with $\cL_+$ or  $\cL_-$.
We will systematically abuse of notation and simply say that $x \in \cL_+$ or
$x \in \cL_-$. 

We define the \emph{multiplicity $\mu_x$} of a puncture $x$ as the intersection multiplicity of the branch $\cE_x$  corresponding to $x$ with the lines $\cL_+, \cL_-$.

As a direct consequence of Bezout's Theorem and Lemma~\ref{l-degree}, we have:

\begin{lemma}
  \label{bezout-l}
  For all $p \ge 3$,
  $$\dfrac{\eta'(p)}{3} = \sum_{x \in \cL_+} \mu_x = \sum_{x \in \cL_-} \mu_x,$$
  where the sums are taken over $ x \in \widehat{\cS_p} \setminus \cS_p$. 
\end{lemma}





\section{Asymptotics of Transversality}
\label{s:asymptotics}

Following~\cite{AKMCubic},  in
\S \ref{proof-modulo-ss} we introduce a meromorphic one-form $d \tau$ on $\widehat{\cS_p}$ which is holomorphic in $\cS_p$. The Euler characteristic of $\widehat{\cS_p}$ is the difference between the number of poles and zeros of $d \tau$, counted with multiplicities. The key to compute this difference will be our Main Lemma  stated in \S \ref{main-lemma-ss}.
Then we prove Theorem A in
\S \ref{proof-modulo-ss}, assuming the Main Lemma.

\subsection{Statement of the Main Lemma}
\label{main-lemma-ss}
For $(a,b) \in \C^* \times \C$, and for all $n \ge 0$, let
$$\omega_n (a,b) : = f^{\circ n}_{a,b} (0).$$
Partial derivatives will be denoted by
$\partial_a \omega_n$ and $\partial_b \omega_n$.
 In dynamical space, we are concerned with the behavior of
$$ \dfrac{d f^{\circ p-1}_{a,b}}{dz} (\infty),$$
under the agreement that the above quantity is the derivative of
$f^{\circ p-1}_{a,b} (1/z)$ at $z=0$.
Then, $\partial_b \omega_p$ and $ \dfrac{d f^{\circ p-1}_{a,b}}{dz} (\infty)$
are meromorphic (rational) functions from $\widehat{\cS}_p$ to $\oC$.
Given a meromorphic function $h: \widehat{\cS}_p \to \oC$,
we denote by $\ord_y h \in \Z$ its \emph{order} at $y \in 
\widehat{\cS}_p$. The order at a zero of $h$
is positive and at a pole is negative.

The key to compute the Euler characteristic
is to understand the asymptotics
of $\partial_b \omega_p$ at the punctures. 

\begin{lemma}[Main]
  \label{l-main}
  Let $x \in \widehat{\cS_p}$ be a puncture of $\cS_p$.
  Then:
  \begin{equation}
    \label{eq:main-l}
    \ord_x \partial_b \omega_p  =      \ord_x \dfrac{d f^{\circ p-1}_{a,b}}{dz} (\infty) + \begin{cases}
    0  & \text{ if } x \in \cL_+, \\
    2 \mu_x  & \text{ if } x \in \cL_-.\\
  \end{cases} 
\end{equation}
\end{lemma}

The lemma is proven in \S \ref{s:puiseux}.

\subsection{Meromorphic One-form and Euler Characteristic}
\label{proof-modulo-ss}

The curve $\cS_p$ is contained in the vanishing locus
of $\omega_p(a,b)$. The gradient of $\omega_p$ is non-vanishing
in $\cS_p$ (see \cite{LWVcor} \cite{EpsteinTransversality} ). Thus, the differential of $\omega_p$ yields a
meromorphic one-form in $\widehat{\cS_p}$ as follows.

\begin{proposition}
  For all $(a,b) \in \cS_p$,
  $$(\partial_a \omega_p (a,b), \partial_b \omega_p (a,b)) \neq (0,0).$$
  Moreover,
  $$ d \tau :=
  \begin{cases}
    \dfrac{da}{\partial_b \omega_p (a,b)} & \text{ if } \partial_b \omega_p (a,b) \neq 0, \\ \\
    - \dfrac{d b}{\partial_a \omega_p (a,b)} & \text{ if } \partial_a \omega_p (a,b) \neq 0,
  \end{cases}$$
  is a well defined non-vanishing holomorphic one-form in $\cS_p$ which extends to a meromorphic one-form in $\widehat{\cS_p}$.
\end{proposition}

We omit the proof of the proposition which is a direct consequence of the transversality results cited above.

\begin{proof}[Proof of Theorem~\ref{A}, assuming Lemma~\ref{l-main}]
  We prove the formula for the Euler characteristic.
  To simplify notation
  let $$G (a,b) := \dfrac{d f^{\circ p-1}_{a,b}}{dz}(\infty).$$
  The number of zeros and  poles of
  the rational map $G:\widehat{\cS}_p \to \oC$ agree, taking into account multiplicities.
By the chain rule, the zeros of $G$ in $\cS_p$ are located at the parameters where the free critical point
is in the orbit of $\omega=0$.
That  is, the zeros of $G$ are the centers of type II components of period $p$.
We postpone proving that 
all the zeros of $G$ are simple to the lemma below.
Since the rest of the poles and zeros of $G$ are at the punctures, we have:
  $$\sum_{x \in \widehat{\cS}_p \setminus \cS_p} \ord_x G + \eta_{II}'(p) = 0.$$
  
  Now we compute the total degree of the one-form $d \tau$. Let $x$ be a puncture of multiplicity $\mu_x$. 
  If   $x \in \cL_\pm$, then there exists a local uniformizing parameter $s$  such that $s^{\pm \mu_x} =a$. Therefore,
  $$da = \pm \mu_x {s^{\pm \mu_x-1}}ds.$$
  Hence,
  $$\ord_x d\tau = -1 \pm \mu_x - \ord_x \partial_b \omega_p,$$
  In view of Lemma~\ref{l-main},
\begin{eqnarray*}
  \sum_{x \in \cL_+} \ord_x d\tau & = & -\sum_{x \in \cL_+} 1 + \sum_{x \in \cL_+} \mu_x - \sum_{x \in \cL_+} \ord_x G, \\
  \sum_{x \in \cL_-} \ord_x d\tau & = & -\sum_{x \in \cL_-} 1 - \sum_{x \in \cL_-} \mu_x - \sum_{x \in \cL_-} \ord_x G - 2 \sum_{x \in \cL_-} \mu_x.
\end{eqnarray*}

From Lemma~\ref{bezout-l}, the sum of the multiplicities of the punctures in each line is
$\eta'(p)/3$ and we obtain:
$$-\chi(\widehat{\cS_p})= \sum_x \ord_x d \tau = -N_p -\dfrac{2 \eta'(p)}{3}+ \eta_{II}'(p).$$
The desired formula for $\chi(\cS_p)$ follows.
\end{proof}

\begin{lemma}
  Let $G: \cS_p \to \oC$ be as above.
  Then, the zeros of $G$ are simple.  
\end{lemma}

\begin{proof}
  Let $\nu=\nu(a,b):=-2a/b$ be the free critical point of $f_{a,b}$
  and  $\nu_j := f^{\circ j}_{a,b} (\nu).$
  Given a zero $c=(a,b)$ of $G$, there exists
  $1 \le j < p$ such that
  $\omega_j(a,b) = \nu(a,b)$.
  Close to $c$, there exists a non-vanishing  holomorphic function
  $u(a,b)$ such that $$G(a,b)=u(a,b) \cdot f'_{a,b}(\omega_j(a,b)).$$
  Since $\nu(c)$ is a simple critical point of $f_c$, 
  the multiplicities of $c \in \cS_p$ as a zero of $\omega_j - \nu:\cS_p \to \C$
  and as a zero of $G$, coincide. According to~\cite{LWVcor}, the intersection
  at $c$ of the curves  
  $\omega_j-\nu=0$ and $\cS_p$ (i.e. $\omega_p-\omega_0=0$) is transversal.
  Hence, $c \in \cS_p$ is a simple zero of $\omega_j-\nu:\cS_p \to \C$, and therefore, it is a simple zero of $G$. 

\end{proof}

\section{Parabolic rescalings}
\label{s:parabolic}
Our asymptotics of transversality lemma
is closely related to \emph{rescaling limits}.
Following Rees and Stimson~\cite{Stimson}, as one approaches a puncture along $\cS_p$, an
appropriate iterate in appropriate coordinates, converges to a parabolic
quadratic rational map. This phenomenon was notably exploited by Epstein in~\cite{EpsteinBounded}, and
further studied by De Marco~\cite{DeMarcoQuadratic},
Nie and Pilgrim~\cite{NiePilgrimBounded}, Luo~\cite{LuoTree}, and the author~\cite{KiwiRescaling}, among others.
The aim of this section is to state related results,
all of them  well known, which will be used
to prove the Main Lemma.

Consider
\begin{eqnarray}
  \label{eq:cfamilies}
  R_v(z) &:=& v + \dfrac{z^2}{z+1}, \quad v \in \C.\\
  Q_c(z) &:=& z^2 +c, \quad c \in \C.
\end{eqnarray}
The family $\{R_v\}$ is a parametrization of $\per_1^{cm}(1) \subset \cM_2^{cm}$,
the curve formed by critically marked
quadratic rational maps with a multiple fixed point.
The maps $R_v$ are normalized so that the multiple fixed point is  $z=\infty$ and $z=-1$ maps to $z=\infty$. We regard $\omega=0$ as the marked critical point. Note that the other critical point
of $R_v$ is $\omegap:=-2$.

The quadratic family $\{Q_c\}$ is $\per_1^{cm}(0)$, the curve formed by critically marked quadratic rational maps with a super-attracting fixed point. For $Q_c$, the superattracting fixed point is  at $\infty$ and
the marked critical point $\omega=0$.

The Main Lemma is, in a certain sense, a manifestation of well known
transversality results for the families $\{R_v\}$ and $\{Q_c\}$.
This manifestation is through (parameter) rescaling limits. 
A version of the following result is 
contained in Stimson's thesis~\cite{Stimson}, its proof is also sketched in
\cite[Section 7.4]{ReesAsterisque} (c.f. ~\cite{EpsteinBounded,DeMarcoQuadratic}). 

\begin{theorem}
  Consider a puncture $x$ of $\cS_p $.
  Then there exist $q_x \ge 2$, $v_x \in \C$
  and a holomorphic family of M\"obius transformation $\{M_f\}$ parametrized by $f \in \cE_x^*$ such that  $M_f(\omega)=0$ and
  $$M_f^{-1} \circ f^{\circ q_x} \circ M_f \to R_{v_x}, \text{ as } f \to x,$$
  uniformly in compact subsets of $\C \setminus \{-1\}$.
    Moreover, under iterations of $R_{v_x}$, one of the following holds:
  \begin{itemize}
  \item the critical point $\omega=0$ of $R_{v_x}$ has period $p/q_x$.
  \item the critical point $\omega=0$ of $R_{v_x}$ eventually maps onto the multiple fixed point $\infty$.
  \end{itemize}
\end{theorem}



The aim of this section is to state the well known parameter space
version of the above result, also contained
in Stimson's thesis. However, for the purpose of our calculations,
we need to explicitly write the dependence of $M_f$ and of $v_x$ on
the asymptotics of $\cE_x^*$ at $\cL_\pm$.
There is no claim to originality of these results.

A detailed study of the case in which $\omega$ is eventually
fixed by $R_{v_x}$ will be discussed in \S \ref{puiseux-dynamics-ss}. It is closely related
to rescaling limits which are quadratic polynomials $Q_c$
for which $\omega=0$ is periodic.

\subsection{Fixed Point Multipliers}
Generically a quadratic rational map has $3$ fixed points.
According to Epstein~\cite[Proposition~1]{EpsteinBounded} (
c.f \cite[Lemma 4.1]{MilnorQuadratic}), as $f \in \cS_p$ converges to a puncture,
one of the fixed point multipliers converges to $\infty$, the other
two converge to conjugate roots of unity of order $q \ge 2$.
Observe that as $(a,b) \equiv [a:b:1]$ converges to $[1:\gamma: 0]\in \cL_-$,
we have convergence of  $a^{-1/2}f_{a,b}(a^{1/2} z)$ to $1/z$,
uniformly in $\C^*$.
Therefore, $f_{a,b}$ has a pair of fixed point multipliers converging to $-1$.
As $(a,b) \to (0,b) \neq (0,0)$, we have convergence of $f_{a,b}$ to
$1 +b/z$, uniformly in  $\C^*$.
Thus, the fixed point multipliers of $f$
converge to $\infty$ and to the fixed point multipliers 
of $1+b/z$. If the map $z \mapsto 1+b/z$ has finite order, then its order is
at least $3$.
A straightforward calculation, (e.g. see \cite[Lemma~3.4]{KiwiRees}) yields the following:

\begin{lemma}
  \label{l:ctheta}
  Consider $\theta \in \QS$ of order $q \ge 3$.
  The fixed point multipliers of
  $z \mapsto 1 +b/z$ are $\exp(\pm 2 \pi i \theta)$ if
  and only if  $$b =-(2 \cos \pi \theta)^{-2}.$$
\end{lemma}
For $\theta \in \Q/\Z$ of order at least $3$, let
$$c_\theta := -(2 \cos \pi \theta)^{-2}.$$
 Every point in $\ocS_p \cap \cL_+$  is of the form $[0:c_\theta:1]$ for
some $\theta \in \QS$ of order at least~$3$.

\subsection{Parabolic Parameter Rescalings}
Loosely, the term ``parameter rescalings'' refers to families of maps that arise in suitable partial compactifications of moduli space.
For example, De Marco~\cite{DeMarcoQuadratic} has shown that
the resolution of moduli space iterate maps (for quadratic rational maps) is achieved exactly in the compactification of moduli space by
all \emph{parabolic} rescaling limits. 
A nice illustration of this phenomena comes from a straightforward computation which we leave to the reader:

\begin{proposition}
  Assume that $(a,b) \equiv [a:b:1]$ converges to $[1:\gamma: 0]\in \cL_-$
  with $\gamma \neq 0$, then
  $$\gamma \cdot f^{\circ 2}_{a,b} ({z}/{\gamma}) \to R_{\gamma},$$
  uniformly in compact subsets of $\C \setminus \{-1\}$. 
\end{proposition}

That is, $\cL_-$ \emph{is} identified with $\per_1^{cm}(1) \setminus \cS_1$
via a period $2$ rescaling limit.

The corresponding result for $\cL_+$ can be roughly described as follows. Let $b_0 \in \C$  be such that $z \mapsto 1+b_0/z$
has order $q \ge 3$. If we blow up $\P^2_\C$ once at $[0:b_0:1]$, then the exceptional divisor is naturally identified with  $\per_1^{cm}(1)$ via a period $q_0$ rescaling limit.

\begin{proposition}
  \label{p:stimson}
  Consider $\theta \in \QS$ of order $q \ge 3$.
  Let
  \begin{eqnarray*}
    \beta_\gamma (a) &=& c_\theta + \gamma a + o(a), \quad \text{ as } a \to 0,\\
     M_a (z) &=& a^{-1} c_\theta z.
  \end{eqnarray*}
  Then there exists an invertible affine function
  $v(\gamma)=A\gamma +B$ such that, for all $\gamma \in \C$, as $a \to 0$,
  $$M_a \circ f^{\circ q}_{a,\beta_\gamma (a)} \circ M^{-1}_a \to  R_{v(\gamma)},$$
uniformly in compact subsets of $\C \setminus \{-1\}$.
\end{proposition}

This proposition is a result by Stimson, but in different coordinates.
The required change of coordinates is discussed in \cite{KiwiRees}.
We deduce the proposition from the results in \cite{KiwiRees}.
Equivalent discussions are included in~\cite{DeMarcoQuadratic,EpsteinBounded}.

\begin{proof}
Consider the family
of (totally marked) quadratic rational maps $h_{\rho,\zeta}$, defined for $(\zeta,\rho) \in \C^* \times \C \setminus \{0,-1,-2\}$ by:
\begin{eqnarray*} h_{\zeta, \rho}(z) & :=\zeta z\left(1-\dfrac{2+\rho}{2(1+\rho)} z\right)\left(1-\dfrac{2}{2+\rho} z\right)^{-1} \\ & =\zeta z\left(1-\dfrac{\rho^2 z}{4(1+\rho)\left(1+\frac{1}{2} \rho-z\right)}\right).
\end{eqnarray*}
Here, critical points and fixed points are marked.
Indeed, $h_{\zeta,\rho}$ has a  
fixed
point at $z=0$ with multiplier $\zeta$,
another fixed point at $z=\infty$, a critical point at
$1$, and another critical point at
$1+\rho$. For $\rho$ small, the third fixed point is close to $z=1+\rho/2$. Note that $h_{\zeta,\rho} \to \zeta z$, as $\rho \to 0$, uniformly in
compact subsets of $\oC \setminus \{1\}$.

There exists a sufficiently small a neighborhood $U \subset \C^2$ of $(\exp(2\pi i \theta),0)$ such
that, for all $(\zeta,\rho) \in U$ with
$\rho \neq 0$, we have that $1,  h_{\zeta,\rho}(1)$
and $h^{\circ 2}_{\zeta,\rho}(1)$ are pairwise distinct.
Thus, for all $(\zeta,\rho) \in U$
with $\rho \neq 0$, there exist  unique 
$a=a(\zeta,\rho)$ and $ b=b(\zeta,\rho)$
such that $h_{\zeta,\rho}$ is conjugate to $f_{a,b}$,
via the M\"oebius conjugacy mapping $1,  h_{\zeta,\rho}(1)$
and $h^{\circ 2}_{\zeta,\rho}(1)$ to $0,\infty,1$, respectively.
It is not difficult to conclude that 
the map $(\zeta,\rho) \mapsto (a(\zeta,\rho),b(\zeta,\rho))$ extends to a biholomorphic map $H$
from $U$ onto a neighborhood $U'$ of $(0,c_\theta)$ (e.g. see the proof Lemma~3.7 in \cite{KiwiRees}).
Under $H$, the line $\rho=0$ maps into $a=0$.

According to Stimson (Theorem~3.6 in \cite{KiwiRees}), for any $m \in \C$, if
$$\zeta = \exp(2 \pi i \theta) (1 + m \rho) + o(\rho).$$
then
$$\dfrac{h^{\circ q}_{\zeta, \rho} (1+z\rho)-1}{\rho} \to qm + z + \dfrac{1}{4(z-1/2)}$$
uniformly in compact subsets of $\C \setminus \{1/2\}$.
Note that the limit is conjugate to $R_v(w)$ with
$v=-2 q m + 1$ via the change of coordinates $z=w/2$.
Using the derivative of the biholomorphic map $H$, given $\gamma$,
there exists an invertible 
map $m(\gamma):=\alpha \gamma +\beta$ such that
 if
$$b=c_\theta + \gamma a +o(a),$$
then
$$\zeta(a,b) = \exp(2 \pi i \theta) (1 + m(\gamma) \rho(a,b)) + o(\rho(a,b)),$$
as $a \to 0$.
Hence,  there exists $M_{a,b}$ such that
$$M_{a,b} \circ f_{a,b}^{\circ q} \circ M_{a,b}^{-1} \to R_{v(\gamma)}$$
as $a \to 0$ 
where $$v(\gamma):=-2 q m(\gamma) +1.$$
Moreover, note that  $M_{a,b}$ can be chosen 
such that $M_{a,b}(0)= 0$,
$M_{a,b} (-2a/c_\theta) = -2$, and $M_{a,b}(\infty) = \infty$.
Thus, we may choose $M_{a,b} (z)=a^{-1}c_\theta z=M_a(z)$.
\end{proof}

\subsection{Local parametrization}
\label{ss:local-par}
To compute the order of various quantities at a puncture $x$ it
is convenient to work with a local parametrization of $\cE_x$.

Denote by $\pi_a: \C^* \times \C \to \C$
the projection onto the first coordinate. Then the multiplicity of $x$, denoted $\mu_x$, is the degree of the \emph{covering} $\pi_a : \cE_x^* \to D^*$ where
$D^*$ is a sufficiently small punctured disk neighborhood of $0$ or $\infty$ and, $\cE_x$ is a suitable representative of the branch corresponding to $x$.
Note that $D^*$ is a neighborhood of $0$ if $x \in \cL_+$, and
of $\infty$ if $x \in \cL_-$.
Since all degree $\mu_x$ coverings of a punctured disk are equivalent, it follows that there exists $\varepsilon >0$ and a power series:
$$b_x(s):=\sum_{j \ge j_0} c_j s^{j},$$
convergent in $\D_\varepsilon^*$ such that
$$\begin{array}{ccc}
  \D_\varepsilon^*&\to &\cE_x^*\\
  s &\to &  (s^{\pm \mu_x},  b_x(s))
\end{array}
$$
is a conformal isomorphism, for a choice of $\cE_x$.
The series $b_x(s)$ is unique up to a $\mu_x$-root of unity $\eta$. That is, up to replacement of $c_j$ by $\eta^{\pm j} c_j$, for all $j$.
We say that $(s^{\pm \mu_x}, b_x(s))$ is a \emph{local parametrization of $\cS_p$ at $x$}, and sometimes we will just say that $b_x$ is a local parametrization.

For future reference, we record a direct consequence of
Lemmas~\ref{l-degree} and~\ref{l:ctheta}:

\begin{lemma}
  \label{l-first-order}
  If $(s^{-\mu},b(s))$ is a local parametrization of a branch at $\cL_-$,
  then, as $s \to 0$,  $$b(s) = \dfrac{c}{s^{\mu}} + o(s^{-\mu}).$$

  If $(s^{\mu},b(s))$ is a local parametrization of a branch at $\cL_+$,
  then, as $s \to 0$,
  $$b(s) = c_\theta + c {s^{\mu}} + o(s^{\mu}).$$
\end{lemma}

\begin{proof}
  In the first case $[s^{-\mu}:b(s):1] = [1: s^{\mu} b(s) : s^\mu] \subset \cS_p$
  for all $s \neq 0$. Hence, as $s\to 0$, we have that $s^{\mu} b(s) \not\to 0$,
  by Lemma~\ref{l-degree}.
  The second case is similar. 
\end{proof}

\subsection{Complex Transversality}

The Main Lemma is, in a certain sense,
a manifestation of the following transversality results for the families
$\{R_v\}$ and $\{Q_c\}$.

\begin{lemma}
  \label{l:transversality}
  Consider $c_0 \in \C$ and $q_0 \ge 1$ such that
  $Q_{c_0}^{\circ q_0} (0) =0$ and $q_0$ is the smallest positive integer with this property. Then
  $$\dfrac{d Q^{\circ q_0}_{c}(0)}{dc} (c_0) \neq 0.$$
  
  Consider $v_0 \in \C$ and $\ell_0 \ge 1$
  such that $R_{v_0}^{\circ \ell_0} (0) = 0$ or $-1$ and $\ell_0$ is the smallest positive integer
  with this property. Then
  $$\dfrac{d R^{\circ \ell_0}_{v}(0)}{dv} (v_0) \neq 0.$$
\end{lemma}

\begin{proof}
For the quadratic family $\{Q_c\}$, the lemma follows from a result by Gleason or by Douady and Hubbard's parametrization of hyperbolic components  (e.g. see~\cite[Expos\'e XIX]{OrsayNotes}).
For the parabolic family $\{ R_v\}$, if $\omega=0$ is
periodic under $R_{v_0}$, then Douady and Hubbard's parametrization of hyperbolic
components applies.
If $R_{v_0}^{\circ \ell_0} (0) = -1$, then we have
the critical orbit relation
$R_{v_0}^{\circ \ell_0+2} (0)= R_{v_0}^{\circ \ell_0+1} (0)$.
Consider the critical orbit relation curve  $\cX$ contained in
$\cM^{cm}_2$ formed by all maps
such that the marked critical
point maps in $\ell_0+1$ to a fixed point. According to~\cite{LWVcor},
$\cX$ is smoothly embedded near $R_{v_0}$. Note that $R_{v_0} \in \operatorname{Per}^{cm}_1(1) \cap \cX$.
It is sufficient
to show that the intersection is transversal. Otherwise,
for $\lambda <1$, sufficiently close to $1$,
$\operatorname{Per}^{cm}_1(\lambda) \cap \cX$ has at least two parameters
close to $R_{v_0}$, where $\operatorname{Per}^{cm}_1(\lambda)$
is the curve of maps with fixed point of multiplier $\lambda$.
As an application of~\cite{HaTanPinching}, these two parameters are hybrid equivalent in a neighborhood
of the Julia set and conformally equivalent in the complement. Hence, these
two elements of $\cM_2^{cm}$ are in fact the same. This contradiction
shows that the intersection is transversal.
\end{proof}




\section{Puiseux Series Dynamics}
\label{puiseux-dynamics-ss}
\label{s:puiseux}


An algebraic closure of the field of \emph{ formal Laurent series} $\laurent$ with coefficients in $\C$ is the \emph{field of formal Puiseux series} $\puiseux$ (e.g. see~\cite{LibroCasas}). One may realize $\puiseux$ as  the injective limit, with respect to the obvious inclusions, of the fields $\laurentm$ for  $m \in \N$.
Given a non-zero element
$$z = \sum_{n \geq n_0} c_n t^{n/m} \in \puiseux,$$
its \emph{order at $t=0$} is
$$\ord_0 z := \min \{ n/m \ge n_0/m : c_n \neq 0 \}.$$
It endows $\puiseux$ with the  non-Archimedean absolute value:
$$\val{z} := \exp(-\ord_0 z).$$

Local parametrizations of  plane algebraic curve germs
are closely related to Puiseux series.
In our case, consider a puncture $x$ with multiplicity $\mu_x$ and
a local parametrization $(s^{\pm \mu_x},b_x(s))$
where
$$b_x(s)= \sum_{j\ge j_0} c_j s^j.$$
Then the series
$$\beta_x (t) := \sum_{j \ge j_0} c_j t^{j/\mu_x} \in \C((t^{1/\mu_x})) \subset \puiseux$$
is called \emph{a Puiseux series for the branch $\cE_x$}.
A Puiseux series for $\cE_x$ is unique up to multiplication by a
$\mu_x$-root of unity.
Namely, given such a root of unity $\eta$, the series with coefficients
$\eta^j c_j$ is also a Puiseux series for $\cE_x$. All the Puiseux
series for $\cE_x$ are of this form.

Regard rational functions $\varphi \in \puiseux (z)$ as dynamical
systems acting on the projective line over $\puiseux$.
The Puiseux series for branches of $\cS_p$ at $\cL_\pm$ are dynamically defined
parameters in a  family of quadratic rational
maps $\varphi_\beta \in \puiseux (z)$.
More precisely,
given $\beta \in \puiseux$,
  let 
  $$\varphi_\beta (z) :=1 + \dfrac{\beta}{z} + \dfrac{t}{z^2} \in \puiseux (z), \quad \text{ if } \val{\beta} \le 1,$$
and  
$$\varphi_\beta (z) :=1 + \dfrac{\beta}{z} + \dfrac{1}{t z^2} \in \puiseux (z),
\quad \text{ if } \val{\beta} > 1.$$

Note that $\upomega :=0$ is a critical point of $\varphi_\beta$, for all $\beta$.
The other critical point is $\upomega':= -2t/\beta$ if $\val{\beta}\le1$
and $\upomega':=-2/t\beta$,  otherwise.

\begin{lemma}
  \label{l:periodic-correspond}
  A branch of $\cS_p$ at $\cL_\pm$ has  
  Puiseux series $\beta_0$ 
    if and only if the critical point $\upomega$
   has period exactly $p$ under iterations of
   $\varphi_{\beta_0}$. 
\end{lemma}

We postpone the proof of this lemma to \S \ref{s:puiseux2complex} and
proceed towards stating 
the Puiseux series version of the Main Lemma \ref{l-main}.

For a rational map $\psi \in \puiseux (z)$ such that $\psi(\infty) \neq \infty$,
by definition, the derivative of $\psi$ at $\infty$ will be the derivative
of $\psi(1/z)$ at $z=0$.

Given $\beta$, we write
$$\upomega_j (\beta):= \varphi_\beta^{\circ j} (0).$$

Different ``normalizations'' arise according to whether $\val{\beta} \le 1$
or not. It is therefore convenient to introduce:
$$ \tau(\beta) :=
\begin{cases}
  t & \text{ if } \val{\beta}>1,\\
  1 & \text{ if } \val{\beta}\le 1.
\end{cases}
$$
When clear from context, to lighten notation, we simply write $\tau$ for
$\tau(\beta)$.

Below, in \S \ref{s:puiseux2complex}, we show that our next result is equivalent to the Main Lemma~\ref{l-main}. 

\begin{lemma}[Main Lemma, {\it \'a la Puiseux}]
  \label{main-puiseux-l}
  Assume ${\beta_0} \in \puiseux$ 
  is such that $\upomega$ has period $p$ under $\varphi_{{\beta_0}}$. 
  Then,
  \begin{equation}
    \label{eq:main-puiseux-l}
    \ord_0 \dfrac{d\upomega_p}{d\beta}({\beta_0}) =
    \ord_0 \dfrac{d\varphi^{\circ p-1}_{{\beta_0}}}{dz}(\infty)+ 2 \ord_0 \tau.
    \end{equation}
\end{lemma}

We will prove \eqref{eq:main-puiseux-l} working our way towards the $p$-th iterate,
establishing a similar formula for some lower iterates.
It will require some knowledge  about the non-Archimedean dynamics of
$\varphi_{\beta_0}$, discussed in \S \ref{s:L} and \S \ref{s:q-dynamics}. 
The substantial part of the proof is contained in
\S \ref{s:bounds} through  \S \ref{s:primitive-r}.

\subsection{From Puiseux series to complex numbers}
\label{s:puiseux2complex}
Let us first prove that parameters $\beta \in \puiseux$ such that
the critical point $\upomega$ is periodic under $\varphi_\beta$ are in correspondence with parametrizations of branches of periodic curves at $\cL_\pm$.
The proof is essentially an immediate consequence of the fact that formal
solutions $\beta \in \puiseux$ to the equation  
$f_{t^{\pm 1},b}^{\circ p}(0)=0$, in the unknown $b$, are automatically
``convergent'' (e.g. see \cite[\S 7.2]{FischerPlaneCurves}). 

\begin{proof}[Proof of Lemma \ref{l:periodic-correspond}]
  It is sufficient to show that $\beta_0$ is the Puiseux series of a branch
  of a periodic curve with period dividing $p$ if and only if
  $\upomega=0$ is periodic of period dividing $p$ under $\varphi_{\beta_0}$.
  
Assume that $(s^{\pm \mu}, b(s))$ is a local parametrization of
$\cE_x$ for some puncture $x$ at $\cL_\pm$ of a periodic curve of period dividing $p$.
Then $f_{s^{\pm \mu}, b(s)}^{\circ p} (0) =0,$
  where
  $$f_{s^{\pm \mu}, b(s)} (z) = 1 + \dfrac{b(s)}{z} + \dfrac{s^{\pm \mu}}{z^2}\in \C((s))(z).$$
  Via the isomorphism $\bar{\iota}: \C((s))(z) \to \C((t^{1/\mu}))(z)$ that sends $s$ to $t^{1/\mu}$,
  we have that $f_{s^{\pm \mu}, b(s)}(z)$ becomes $\varphi_{\beta_0}(z)$,
  where $\beta_0 = b(t^{1/\mu})$. Thus, $\varphi_{\beta_0}^{\circ p} (0) =0$.
  Hence, $\upomega=0$ is periodic with period dividing $p$, under iterations of
  $\varphi_{\beta_0}$.

  Conversely,  assume that  $\upomega =0$
  has period dividing $p$, under $\varphi_{\beta_0}$, for some
  $\beta_0\in \puiseux$.
  Let $\mu \ge 1$
  be minimal such that $$\beta_0 = \sum_{j \ge j_0} c_j t^{j/\mu} \in \C((t^{1/\mu})).$$ 
   Via $\bar{\iota}^{-1}$,
   we have that $\varphi_{\beta_0}(z)$ becomes
  $f_{s^{\pm \mu},b_0(s)}(z)$
where $$b_0(s) = \sum_{j \ge j_0} c_j s^j \in \C((s)).$$
 Hence, $(s^{\pm \mu}, b_0(s))$
  is a formal solution of $f_{s^{\pm \mu},b(s)}^{\circ p}(0) =0$. Every formal
  solution is convergent, for $|s|$ sufficiently small (e.g. see \cite[\S 7.2]{FischerPlaneCurves}).
  Since $\mu$ is minimal, $(s^{\pm \mu}, b_0(s))$ is a parametrization
  of $\cE_x$ for some puncture $x$ of a periodic curve with period dividing~$p$.
\end{proof}

Now we show that the Puiseux series version of the Main Lemma (i.e. Lemma~\ref{main-puiseux-l}) is equivalent to the Main Lemma (i.e. Lemma~\ref{l-main}).
Given a puncture $x$ with local parametrization $(s^{\pm \mu}, b_x(s))$,
any meromorphic function $w:\cE_x \to \oC$ has a Laurent
series expansion in the local parameter $s$. That is,
$w$ is an element of $\C((s))$.
Under the morphism determined by
$s \mapsto t^{1/\mu}$ between $\C((s))$ and $\C((t^{1/\mu}))$,
the order at $0$ is divided by $\mu$.
This simple observation leads to the equivalence between the
two versions of the Main Lemma:

\begin{lemma}
  \label{l:puiseux-complex}
  Let $\beta_0$ be a Puiseux series for a puncture $x$.
  Then,  \eqref{eq:main-l} holds if and only if
  \eqref{eq:main-puiseux-l} holds.  
\end{lemma}

For simplicity, we use partial derivatives notation.
In particular, the derivatives involved
in~\eqref{eq:main-puiseux-l}, will be denoted by $\partial_\beta \upomega_p (\beta_0)$ and $\partial_z \varphi^{\circ p-1}_{\beta_0} (\infty)$.

\begin{proof}
    We drop the subscript from the multiplicity 
    $\mu_x$ and simply write $\mu$.
    Denote by $\iota: \C((s)) \to \C((t^{1/\mu}))$ the isomorphism defined
  by $s \mapsto t^{1/\mu}$.
  Consider a local parametrization $(s^{\pm \mu}, b_0(s))$ at $x$ with
  associated Puiseux series $\beta_0$. 
   Then $\iota (b_0) =\beta_0$.
  Given $z \in \C((s))$ (resp. $\iota(z) \in
  \C((t^{1/\mu}))$) denote by $\ord_0 z$ (resp. $\ord_0 \iota(z)$), the order
  at $s=0$ (resp. $t=0$).
Then,
  obviously, $\ord_0 z = \mu \cdot \ord_0 \iota (z)$.

  Let $\operatorname{ev_s}: \C[a,b] \to \C((s))$ be evaluation
  at $a=s^{\pm \mu}, b =b_0(s)$ and 
  $\operatorname{ev}_t: \C[a,b] \to \puiseux$ be evaluation at
   $a=t^{\pm 1}, b = \beta_0$.
  Then, $\operatorname{ev}_t = \iota \circ \operatorname{ev}_s$.

  Now $\varphi_{\beta_0}(z) \in \puiseux (z)$
  is the image of $f_{a,b}(z) \in \C[a,b](z)$ under the morphism (induced
  by)  $\operatorname{ev}_t$. Therefore,
  $\varphi_{\beta_0}^{\circ n}(z)$ is the image of $f^{\circ n}_{a,b}(z)$,
  under $\operatorname{ev}_t$.
  Also,  $\operatorname{ev}_t$ 
  maps $\partial_b \omega_p$ and 
  $\partial_z f^{\circ p-1}_{a,b}(1/z)$
  onto $(\partial_\beta \upomega_p)(\beta_0)$ and
  $\partial_z \varphi_{\beta_0}^{\circ p-1}(1/z)$, respectively.
Similarly, $\operatorname{ev}_s$ maps $\partial_b \omega_p$ and 
$\partial_z f^{\circ p-1}_{a,b}(1/z)$ onto the Laurent series around $x$,
in the local parameter $s$, of the meromorphic functions $\partial_b \omega_p(a,b)$
and $\partial_z f^{\circ p-1}_{a,b}(1/z)$ for $(a,b) \in \overline{\cS}_p$.
Since $\operatorname{ev}_t = \iota \circ \operatorname{ev}_s$ and $\ord_0 z = \mu \cdot \ord_0 \iota (z)$, it follows that equation
 \eqref{eq:main-l} is equation
  \eqref{eq:main-puiseux-l},
multiplied by $\mu$. 
\end{proof}

\subsection{A complete field}
\label{s:L}
Recall that $\puiseux$
is endowed with the non-Archimedean absolute value
$$ \val{z} := \exp(-\ord_0 z).$$
The metric induced by $\val{z}$ in
$\puiseux$ is particularly
suited to understand the orders of dynamically defined quantities
around the punctures of $\cS_p$.
However, endowed with this metric, $\puiseux$ is not complete.
We let $\L$ be the field obtained as the completion of $\puiseux$. It follows that $\L$ is also algebraically closed (e.g. see~\cite{LibroCassels}). 
Moreover, each element $z$ of $\L$
can be represented by a series of the form:
$$z = \sum_{n \geq 0} c_n t^{\lambda_n},$$
where $\lambda_n \in \Q$ and $\lambda_n \to +\infty$ as $n \to +\infty$.
Both, $\ord_0 z$  and  $\val{z}$
naturally extend to $\L$.
The value group $\lval \L^\times \rval$ of $\L$ is $\exp(\Q)$.

We regard rational maps $\varphi \in \L(z)$ as dynamical
systems acting on the projective line $\ponel$ over $\L$, and sometimes on the Berkovich projective line $\poneberk$ over $\L$.
For our purpose here, passing from $\puiseux$ to $\L$ is not essential.
However, non-Archimedean dynamics is better understood with the aid of Berkovich
spaces over \emph{complete} fields. For background about non-Archimedean dynamics see~\cite{BakerRumelyBook,BenedettoBook}.
In fact, we intensively employ the full description for quadratic rational dynamics over $\L$ given by the author in~\cite{KiwiPuiseuxQuadratic}.
Although we summarize the main facts about quadratic dynamics in
\S \ref{s:q-dynamics}, for a detailed exposition we refer the reader to~\cite{KiwiPuiseuxQuadratic}.

In \S~\ref{s:spherical}-\S~\ref{s:action-balls}, we summarize basic
facts about non-Archimedean analysis that will be used frequently in the
rest of the paper.

\subsubsection{Spherical metric}
\label{s:spherical}
The projective line $\ponel$ is identified with
$\L \cup \{ \infty \}$ via $[z:1]\mapsto z$ and
$[1:0] \mapsto \infty$. In $\ponel$ we work
with the \emph{spherical metric} which, for $z, w \in \L$, is given by
$$\dist_s (z,w) :=
\begin{cases}
  \val{z-w} & \text{ if } \val{z},\val{w} \le 1,\\
  \val{1/z - 1/w} & \text{ if } \val{z},\val{w} > 1,\\
  1 & \text{ otherwise}.
\end{cases}
$$
It naturally extends to $z=\infty$ or $w=\infty$.

In dynamical space, we will compute the size of derivatives in the \emph{spherical metric}
and denote them by $\| \cdot \|$. For example,
$\|\partial_z \varphi_\beta (\infty)\|=\val{\beta}$, for all $\beta \neq 0$.
In parameter space, we work in the standard coordinate of $\L$.
For example, 
$\| \partial_\beta \varphi_\beta (z) \|$ is
$\val{\partial_\beta (1/\varphi_\beta(z))}$ if $\val{\varphi_\beta(z)}>1$.
In this language, the Main Lemma states that for $n=p$:
\begin{equation}
  \label{eq:nidentity}
\|\partial_\beta \upomega_n(\beta_0)\| = \|\partial_z \varphi_{\beta_0}^{\circ n-1}(\infty)\| \val{\tau}^2.  
\end{equation}
Our strategy will be to prove \eqref{eq:nidentity}
for an increasing sequence of values of $n$, finishing with $n=p$.

\subsubsection{Reduction: from Puiseux series to $\C$}
Often we will pass from dynamics over $\L$ to dynamics over $\C$ via \emph{reduction}. Intuitively, reduction 
amounts to replace $t$ by $0$.
More precisely, 
the set of elements $z= \tilde{z}+o(1) \in \L$ is \emph{the ring of integers}
$$\mathfrak{O} := \{ z \in \L : \val{z}\le 1 \}.$$
The elements of its
unique maximal ideal $$\mathfrak{M}:= \{ z \in \L : \val{z}<1 \}$$
are of the form $z=o(1)$.
The quotient $\mathfrak{O}/\mathfrak{M}$
is canonically identified with $\C$
via
$$\tilde{z} + \mathfrak{M} \mapsto \tilde{z},$$
for $\tilde{z} \in\C \subset \L$.
The quotient map naturally extends to:
$$\rd : \ponel  \to \ponec,$$
by declaring $\rd(z) =\infty$ when $\val{z} > 1$.
This map is called the \emph{reduction map}.
Note that each spherical open ball of radius $1$ maps onto a point under
reduction. The obvious inclusion $\ponec \hookrightarrow \ponel$ is a right
inverse of $\rd$. 

Although we mostly work with non-homogeneous coordinates, to understand the effect of reduction on rational maps is convenient to work in homogeneous coordinates.
Any rational map $\varphi: \ponel \to \ponec$
can be written as a quotient of degree $d$ homogeneous polynomials $P, Q$
with coefficients in $\mathfrak{O}$
and with at least one coefficient of absolute value
$1$. The polynomials $\tilde{P}, \tilde{Q}$ obtained after
reducing the coefficients of $P$ and $Q$ might have a common factor, say $H$.
Then there exist relatively prime polynomials $\widehat{P},\widehat{Q}$ with complex coefficients and  a maximal degree common divisor $H$ of $\tilde{P}, \tilde{Q}$ 
such that $\tilde{P} = H \cdot \widehat{P}$ and $\tilde{Q} = H \cdot \widehat{Q}$.
We say that:
$$
\begin{array}{rccl}
  \rd \varphi:&\ponec&\to&\ponec\\
     & [z:w]&\mapsto&[\widehat{P}(z,w),\widehat{Q}(z,w)]
\end{array}
$$
is the \emph{reduction of $\varphi$}.
When convenient we denote $\rd(z)$ by $\tilde{z}$ and
 $\rd \varphi$
by $\tvarphi$.

The reduction map semiconjugates
$\varphi: \pone_\L \to \pone_\L$ with
$\rd \varphi: \pone_\C \to \pone_\C$
for all $z \in \pone_\L$ outside
a finite collection disjoint open balls of spherical radius $1$.
In fact, 
$$(\rd \varphi)(\tilde{z}) = \rd(\varphi(z))$$
for all $z \in \ponel$ such that $H(\tilde{z}) \neq 0$.
Moreover, if $\deg \tvarphi \ge 1$, then $\varphi$ maps the ball $\rd^{-1} (\trz)$ onto the ball
$\rd^{-1}(\tvarphi(\trz))$
with degree $\deg_{\trz} \tvarphi$. 

\subsubsection{Action on Balls and Annuli}
\label{s:action-balls}
A ball $B$ in $\L$ is a set of the form
$$\{ z \in \L : \val{z-z_0} < r \} \text{ or } \{ z \in \L : \val{z-z_0} \le r \},$$
for some $z_0 \in \L$ and some $r >0$. In the first case we say that
$B$ is an \emph{open ball}, in the latter, it is a \emph{closed  ball}.
All balls are topologically open and closed. If $r \notin |\L^\times|_0 = \exp(\Q)$, then the open and closed ball of radius $r$ around $z_0$
coincide. If $r \in |\L^\times|_0$, then we say that the corresponding balls are \emph{rational open} or
\emph{rational closed} balls.

An analytic function $\varphi: \mathfrak{O} \to \mathfrak{O}$ is 
power series $\sum c_i z^i$ 
with coefficients in $\mathfrak{O}$ such that $\val{c_i} \to 0$; equivalently,
the series converges for all $z \in \mathfrak{O}$.
The \emph{reduction} $\tilde{\varphi} = \sum
\tilde{c}_i \tilde{z}^i \in \C[\tilde{z}]$ is a polynomial.
Moreover, $\deg \tvarphi \ge 1$ if and only if $\varphi$ is onto.
In this case, $\deg \varphi = \deg \tvarphi$ is the number of preimages under $\varphi$, counting multiplicities, of points in $\mathfrak{O}$.
Moreover, $\varphi$ maps the maximal open ball $\tilde{z} + \mathfrak{M}$
onto $\tvarphi(\tilde{z}) + \mathfrak{M}$
with degree $\deg_{\tilde{z}} \tvarphi$. Note that $\deg \varphi$ is the number
of critical points of $\varphi$, counting multiplicities, plus $1$. This is
true since the residue field $\mathfrak{O}/\mathfrak{M}=\C$ has characteristic
zero.

We identify $\ponel$ with $\L \cup \{\infty\}$.
If $\varphi: \ponel \to \ponel$ has no poles in a ball $B$, say of radius $r$,
then
$\varphi(B)$ is a ball of the same type (rational, open, closed), say of radius $r'$.
Moreover, the Schwarz Lemma holds; that is,  for all $z \in B$:
$$\val{\dfrac{d\varphi}{dz}(z)}\le \dfrac{r'}{r},$$ with equality if and only if $\varphi: B \to \varphi(B)$ is a bijection; equivalently, $$\val{\varphi(z)-\varphi(z')}= \dfrac{r'}{r} \val{z-z'}$$ for all $z,z' \in B$.
Furthermore, since the residue field has characteristic zero, the Mean Value Theorem also holds: if $\varphi:B \to \L$
and $\val{\varphi'(z)} \le C$ for all $z \in B$, then for all $z,z' \in B$,
$$\val{\varphi(z)-\varphi(z')} \le C \val{z-z'}.$$
In general,
a \emph{(resp. rational, open) closed ball} $B$ in $\ponel$ containing $\infty$ is the
complement of a (resp. rational, closed) open ball in $\L$. 
Given a rational map $\varphi: \ponel \to \ponel$ and a ball $B$, then
$\varphi(B)$ is a ball of the same type or $\varphi(B)=\ponel$. 
If $B$ is a ball and
$\varphi(B)$ is a ball, then $\varphi: B \to \varphi(B)$ has a well
defined degree which is the number of critical points in $B$, counting multiplicities, plus $1$. If moreover
$B$ is a closed rational ball, then each maximal open ball in $B$ maps onto a maximal open ball in $\varphi(B)$.

If $\L \supset B \supset B'$ are nested balls, 
then $A=B\setminus B'$ is an \emph{annuli}. Its \emph{modulus}
is $\operatorname{mod}(A):=\log \diam(B) - \log \diam (B')$. These definitions easily generalize to
$\ponel$. If $A$ is an annulus and $\varphi: \ponel \to \ponel$ is a rational map
such that $\varphi(A)$ is again an annulus, then $\varphi: A \to \varphi(A)$ has
a well defined degree $d_A\ge 1$ and $d_A \cdot \operatorname{mod}(A) = \operatorname{mod}(\varphi(A))$.



\subsection{Quadratic rational dynamics over $\L$: limbs}
\label{s:q-dynamics}
In our family $\varphi_\beta$ we restrict our attention to certain
subsets of parameter space that we call 
\emph{limbs}, by analogy with complex dynamics. More precisely, 
given $\theta \in \Q/\Z$  of order $q \ge 2$,
we define the non-Archimedean limb $\cL_\theta \subset \L$ by:
\begin{eqnarray*}
  \cL_{1/2} &:=& \{ \beta = \gamma t^{-1} : \val{\gamma}= 1 \},\\
  \cL_\theta&:=& \{ \beta = c_\theta + \gamma t : \val{\gamma}\le 1 \}, \quad \text{ if } q \ge 3,
\end{eqnarray*}
where $c_\theta := -(2 \cos \pi \theta)^{-2}.$

By Lemmas~\ref{l-first-order} and \ref{l:periodic-correspond}, limbs contain
the parameters of our interest. Indeed, any $\beta_0$ for which $\upomega$ is periodic of period $p \ge 3$ is contained
in some limb.
In this section, we introduce basic results
about the dynamics of $\varphi_\beta$
with focus on parameters that lie in a limb.


Below, we show that all
maps $\varphi_\beta$ in a limb have an \emph{invariant Rivera domain $U_0$}. 
Although Rivera domains are easier to understand in 
Berkovich spacee, we will keep our discussion in the classical
line $\pone_\L$. In this context, an invariant Rivera domain $U_0$ is
the complement of a finitely many 
pairwise disjoint rational closed
balls such that $\varphi: U_0 \to U_0$ is a bijection and $U_0$ is maximal
with respect to these properties.

To describe the dynamics and the invariant domain $U_0$ of $\varphi_\beta$ 
when $\beta$ lies
in an order $q$ limb $\cL_\theta$,  we
define $q$ closed balls
$B_0, \dots, B_{q-1}$, subscripts mod $q$,  and a maximal open ball $D_0$ contained in $B_0$. We do not
make the dependence of $B_j$ on $\theta$ explicit.
Only $D_0$ when $q=2$ also depends on $\beta$.

Given $\beta=\gamma t^{-1} \in \cL_{1/2}$, let
\begin{eqnarray*}
  B_0 &:=& \{ z:\val{z} \le 1 \},\\
  B_1 &:=& \{ z:\val{z}  \ge \val{t}^{-1}\}, \\
  D_0 &:=& \{ z:\val{z-\gamma^{-1}}<1 \}.
\end{eqnarray*}

Given $\theta \in \QS$ of order $q \ge 3$,
consider $\beta = c_\theta + o(1) \in \cL_\theta$.
Recall that $\upomega_j(c_\theta)$ denotes
the $j$-th iterate of $\upomega=0$ under $\varphi_{c_\theta}$.
Let
\begin{eqnarray*}
  B_j &:=& \{ z: \val{z - \upomega_j(c_\theta)} \le \val{t} \}, \quad \text{ if } j=0 \text{ or } 2 \le j < q\\
  B_1 &:=& \{ z:\val{z}  \ge \val{t}^{-1}\}, \\
  D_0 &:=& \{ z: \val{z+c_{\theta}^{-1}t} < \val{t} \}.
\end{eqnarray*}
Note that the critical points of $\varphi_\beta$
lie in $B_0$ and the critical values in $B_1$.

\begin{proposition}
\label{p:level0}
  Consider a parameter $\beta$ in a limb $\cL_\theta$ of order $q$.
Let $$U_0 := \pone_\L \setminus \bigcup_{j=0}^{q-1}  B_j.$$
  Then:
  \begin{enumerate}
  \item $\varphi_\beta : U_0 \to U_0$ is a bijection.
  \item If $j\neq 0$, then $\varphi_\beta : B_j \to B_{j+1}$ is a bijection.
  \item $\varphi_\beta : B_0 \setminus D_0 \to B_1$ has degree $2$.
  \item $\varphi_\beta : D_0 \to \pone_\L \setminus B_1$ is a bijection.
  \end{enumerate}
\end{proposition}

\begin{proof}
  We assume that $q \ge 3$. The case $q=2$ follows from a similar argument.

  Consider a parameter $\beta \in \cL_\theta$. To simplify notation
  write $c:=c_\theta$. Hence $\varphi_c$ is the map corresponding to $\beta =c$.
  Also, $$M(\trz):=1 + \dfrac{c}{\trz} =\rd \varphi_\beta (\trz)  $$
  is an order $q$ M\"oebius transformation.

  The semiconjugacy between $\varphi_\beta$ and $M$, furnished by reduction, holds
  for all $z \notin \mathfrak{M}$. Hence, the open ball $\trz + \mathfrak{M}$
  of radius $1$ is mapped isometrically (in the spherical metric)
  onto $M(\trz) + \mathfrak{M}$, for
  all $\trz \in \oC \setminus \{0\}$. In particular,
  $\upomega_j(\beta) \notin \mathfrak{M}$, for $1\le j< q$.
  Furthermore, a ball of spherical
  radius $\val{t}$ around $\upomega_j (\beta)$ maps onto
  a ball with the same spherical radius around  $\upomega_{j+1} (\beta)$,
  for all $1 \le j < q$. 

   We postpone the proof of (1).  To prove (2), let $\gamma \in \mathfrak{O}$ be such that
   for  $\beta = c + \gamma t \in \cL_\theta$. We claim
  that, for all $2\le j \le q$,
\begin{eqnarray}
  \label{eq:perball}
  \upomega_j (\beta) &=& \upomega_j(c) + O(t).
  \end{eqnarray}
  Indeed, this is trivially true for $j=2$ and we proceed by induction. Suppose true for some $j$ such that $2 \le j< q$. Then
  $$\val{\varphi_\beta(\upomega_j(\beta))-\varphi_\beta(\upomega_j(c))}
  = \val{\upomega_j(\beta) - \upomega_j(c)} \le \val{t},$$
  since $\varphi_\beta$ is an isometry  on open balls of radius $1$ disjoint from $\mathfrak{M}$.
  Moreover,
  $$\val{\varphi_\beta(\upomega_j(c))-\varphi_c(\upomega_j(c))}=
  \val{\dfrac{t \gamma}{\upomega_j(c)}} \le \val{t}.$$
  From the strong triangle inequality, \eqref{eq:perball} holds for $j+1$.
  Therefore, (2)  holds for $1 \le j \le q-2$.
  To prove (2) for $j=q-1$, we must show that $\upomega_q(c)=\upomega+O(t)$.
 In fact, a straightforward induction shows that for $2 \le j \le q$:
  \begin{eqnarray*}
      \upomega_j (c) &=& M^{\circ j-2} (1) + O(t).
  \end{eqnarray*}
Since $ M^{\circ q-2} (1) =0=\upomega$, assertion (2) holds.
  
  For (3) and (4), note that
  $$h(\trz):= \rd \left( {t} {\varphi_\beta (zt)} \right) = \dfrac{c\trz+1}{\trz^2}.$$
  It follows that  $D_0 = \{z: z= -c^{-1} t + o(t)\}$ maps bijectively onto
  $\mathfrak{M}$ by $t \varphi_\beta$, which proves (4).
  For all $\trz \neq -c^{-1}$, an open balls of the form $\trz t + o(t)$ is
  mapped by $\varphi_\beta$ map onto the open ball
  of spherical radius $\val{t}$ centered at $h(\trz)/t$ with
  degree $\deg_{\trz} h$, which yields (3).
  
  To finish, we deduce (1). First observe
  that every element of $\pone_\L \setminus U_0$ has
  two preimages, counting multiplicities, in $\pone_\L \setminus U_0$.
  Thus, $\varphi_\beta (U_0) \subset U_0$. By (4), a point $z \in U_0$ has  exactly one preimage 
  in $D_0$. By (2) and (3), no preimage of $z$ is contained in $B_1, \dots, B_{q-1},
  B_0\setminus D_0$. Hence, $z$ has exactly one preimage in $U_0$. That is, $\varphi_\beta (U_0) = U_0$ and
  $\varphi_\beta$ is injective in $U_0$.
\end{proof}

We collect basic consequences about the action of $\varphi_\beta$ on balls
and annuli:

\begin{corollary}
  \label{c:action-balls}
  Consider $\beta$ in an order $q$ limb. Then the following holds:
  \begin{enumerate}
  \item If $B$ is a ball properly contained in $B_j \setminus D_0$, then
    $\varphi_\beta(B)$ is a ball contained in $B_{j+1}$. Moreover,
    $\varphi_\beta: B \to \varphi_\beta(B)$ has degree $1$ if $B$ is critical point free and degree $2$, otherwise. Furthermore, every maximal open ball
    contained in $B_j$ distinct from $D_0$ maps onto a maximal open ball
    contained in $B_{j+1}$.
  \item For $2 \le j <q$,
    $$C_j := \left(\varphi_\beta|_{D_0}\right)^{-1} (B_{j+1})$$
    is a ball which maps bijectively onto $B_{j+1}$.
  \item
    If $A=B \setminus B'$ is an annulus such that $B$ is properly contained in $B_j \setminus D_0$ or in $C_j$, for some $j$, then $\varphi(A)$ is an annulus. Moreover,
    $\varphi: A \to \varphi(A)$ has degree $2$ if 
    is  $B'$ contains a critical point, otherwise it has degree $1$.
  \item $U_0$ is an invariant Rivera domain.
  \end{enumerate}
\end{corollary}

We omit the proof of the corollary which is an application of \S \ref{s:action-balls} and Proposition~\ref{p:level0}. To prove (4) it is not difficult
to check that $U_0$ is maximal
with respect to the property of being a complement of finitely many balls
such that $\varphi_\beta:U_0\to U_0$ is bijective.

\subsection{Parabolic rescaling}

The balls $B_0, \dots, B_{q-1}$ correspond to a ``parabolic rescaling limit''.
That is, according to ~\cite[Theorems~1, 2]{KiwiPuiseuxQuadratic},  after
a change of coordinates mapping $B_0$ onto the ring of integers $\mathfrak{O}$,
the map $\varphi^{\circ q}_\beta$
reduces to a complex quadratic rational map with a parabolic fixed  point.
We will need to be precise about the
dependence of the complex parabolic map on the parameter $\beta$.
Although one could implement
a non-Archimedean version of
the computations leading to Proposition~\ref{p:stimson}, for simplicity,
we prefer
to employ this  complex dynamics result.

To write down the dependence on $\beta$,
we introduce parametrizations for the limbs
and convenient changes coordinates. 
Given  $\theta \in \Q/\Z$ of order $q \ge 3$,
we let
\begin{eqnarray*}
  \lambda_\theta(\gamma) & := & c_\theta + \gamma t, \,\text{ for } |\gamma|\le 1,\\
  L_\theta (z) &:=& t^{-1} c_\theta z.
\end{eqnarray*}
Also,
\begin{eqnarray*}
  \lambda_{1/2}(\gamma) & := &  \dfrac{\gamma}{t}, \text { for } |\gamma|=1,\\
  L_{1/2} (z)&:=& \gamma z.
\end{eqnarray*}
We abuse of notation and consider the dependence of $L_{1/2}$ on $\gamma$ implicitly. Observe that $\lambda_\theta$ is the natural parametrization of the $\theta$-limb $\cL_\theta$.

Consider
$$
\begin{array}{cccc}
  \pi_0:&B_0& \to& \ponec\\
  & z & \to & \rd(L_\theta(z)).
\end{array}
$$
Then, given $\beta \in \cL_\theta$, the map $\varphi_\beta^{\circ q} : B_0 \setminus D_0 \to B_0$ is semiconjugate, via $\pi_0$, to a complex quadratic rational map $R$ with a parabolic fixed point~\cite[Theorem~1,2]{KiwiPuiseuxQuadratic}.
Moreover, the parabolic fixed point of $R$ is at $\pi_0(U_0)=\infty$ with preimage $\pi_0(D_0)=-1$ and $R$ has a critical point in $\pi_0 (\upomega) =0$.
Hence, $R$ is an element $R_v$ of the parabolic family \eqref{eq:cfamilies} and
$$\pi_0 \circ \varphi_\beta^{\circ q} (z) = R_v \circ \pi_0(z),$$
for $z \in B_0 \setminus D_0$.
Equivalently, $$R_v = \rd(L_\theta \circ \varphi^{\circ q}_{\beta}
  \circ L^{-1}_\theta ),$$ 
where $$v := \rd (L_\theta (\upomega_q(\beta)).$$
The following lemma says that the assignment $\beta \mapsto v$ induces an isomorphism of complex affine lines:

\begin{lemma}
  \label{l:parabolic-family}
  Given $\theta \in \Q/\Z$ of order $q \ge 3$ (resp. $q=2$),
  there exists a non-constant function
  $v_\theta(\tilde{\gamma}) = A \tilde{\gamma} +B \in \C[\tilde{\gamma}]$ (resp. $v_\theta(\tilde{\gamma})=\tgamma$) such that,  for all $\gamma \in \mathfrak{O}$ (resp. $\gamma \in \mathfrak{O} \setminus\mathfrak{M}$),
  $$\rd(L_\theta \circ \varphi^{\circ q}_{\lambda_\theta (\gamma)}
  \circ L^{-1}_\theta ({z})) = R_{v(\tilde{\gamma})}(\tilde{z}) $$
  for all $z \in B_0 \setminus D_0$  or equivalently for all
  $\tilde{z} \in \C\setminus\{-1\}$.
\end{lemma}

\begin{proof}
  When $\theta=1/2$, a straightforward direct computation
  proves the assertion.

If $\theta$ has order $q \ge 3$, then
the function $$\upomega_q : \cL_\theta \to B_0$$ is well defined. Therefore,
$$L_\theta (\upomega_q(\lambda_\theta (\gamma))) \in \mathfrak{O},$$
and
$$w(\gamma):=\rd L_\theta (\upomega_q(\lambda_\theta (\gamma)))$$
only depends on $\tgamma$. That is, $w(\gamma) = v_\theta (\tgamma)$ for
some $v_\theta (\tgamma) \in \C[\tgamma]$.
We must show that $\tgamma \to v_\theta(\tgamma)$ is an affine isomorphism.

According to Proposition~\ref{p:stimson}, there exists
a non-constant complex affine map $\tgamma \mapsto A\tgamma + B$
such that if $\lambda_\theta (\gamma)= c_\theta + \tgamma t$, then 
$$A\tgamma + B = \lim_{t\to 0} t^{-1} c_\theta \omega_q(c_\theta+\tgamma t). $$
Since reduction coincides with the limit as $t \to 0$,
$$ A\tgamma + B = \lim_{t\to 0} t^{-1} c_\theta \omega_q(c_\theta+\tgamma t)=\rd (L_\theta (\upomega_q(c_\theta + \tgamma t)) = v_\theta (\tgamma).$$
\end{proof}

\begin{corollary}
  \label{c:minimal}
  Consider $\beta_0 =\lambda_\theta (\gamma_0)$ in an order $q$ limb $\cL_\theta$.
  Assume that $\upomega$ has period $p$ under $\varphi_{\beta_0}$.
  Then there exists  a minimal $\ell_0 \ge 1$ such that one of the following holds:
  \begin{itemize}
    \item
      $q \ell_0 =p$ and $R_{v_\theta(\tilde{\gamma_0})}^{\circ \ell_0}(0) =0$.
      \item
      $q \ell_0 < p$ and $R_{v_\theta(\tilde{\gamma_0})}^{\circ \ell_0}(0) =-1$.
  \end{itemize}
\end{corollary}

\begin{proof}
  Write $\varphi:=\varphi_{\beta_0}$ and $R_v := R_{v_\theta}(\tgamma_0)$.
Denote by $D$ the maximal open ball contained  in $B_0$ around $\upomega$.
In Lemma~\ref{l:parabolic-family}, 
reduction is a semiconjugacy outside $D_0$.
If the orbit of $\upomega$ is disjoint
from $D_0$, then $D$ is periodic, say of period $\ell_0\ge 1$,
under $\varphi^{\circ q}$. Moreover, $\upomega$ is the unique
periodic point of 
$\varphi^{\circ q \ell_0 } : D \to D$, by Schwarz Lemma.
Therefore, $q \ell_0 =p$ and $D$ is contained in the basin of
attraction of the critical orbit under $\varphi$.
Also, the  critical point $\omega=0$ has
period $\ell_0$ under $R_v$.
If the orbit of $\upomega$ is not disjoint from $D_0$, then there exists a minimal $\ell_0\ge 1$ such that $\upomega_{q\ell_0}(\beta) \in D_0$. In this case, 
$\omega=0$ hits $z=-1$ in $\ell_0$ iterates under $R_v$.
\end{proof}

For a parameter $\beta$ in an order $q$ limb, we say that $\upomega$
has a \emph{non-central return} if $\upomega_{q\ell} (\beta) \in D_0$ for
some $\ell \ge 1$. In this case, we say that $q\ell_0$ is the \emph{first non-central return time} if $\ell_0$ is minimal such that  $\upomega_{q\ell_0} (\beta) \in D_0$.

A parameter without non-central returns is called a \emph{satellite parameter}.

\begin{remark}
  In Berkovich space language, one can phrase Lemma~\ref{l:parabolic-family} as follows.
  Denote by $\per_1(1)$ the space of critically marked complex quadratic
  rational maps with a parabolic fixed point. Note that
  $\per_1(1)$ is an affine line over $\C$.
  Consider an order $q \ge 3$ limb
  $\cL_\theta$, denote by $\lambda_0$ the type II point $\partial \overline{\cL_\theta}$
  and denote by $T_{\lambda_0} \overline{\cL_\theta}$ the space of (finite) directions at $\lambda_0$, naturally endowed with the structure of an affine line over $\C$. Let
  $x_0$ be the type II point $\partial \overline{B_0}$.
  Denote by $T_{x_0} \varphi_\beta^{\circ q}: T_{x_0} \poneberk \to T_{x_0} \poneberk$  the corresponding action on the set of directions at $x_0$.
  Note that $T_{x_0} \varphi_\beta^{\circ q}$ is a well defined conjugacy class
  of maps acting on the $\ponec$-structure of $T_{x_0} \poneberk$.
  Moreover, it belongs to $\per_1(1)$. The lemma, stated in this language,
  says that $\cL_\theta \ni \beta \to T_{x_0} \varphi_\beta^{\circ q} \in \per_1(1)$
  projects to an affine isomorphism  $T_{\lambda_0} \overline{\cL_\theta} \to
  \per_1(1)$. 
  The fact that this map is an isomorphism is of central importance
  to us. 
  We will prove a similar result for ``polynomial
  rescaling limits''. (For $q=2$, with minor modifications, an anologue statement is also true). 
\end{remark}

\subsection{Preliminary bounds}
\label{s:bounds}
The rest of the paper is devoted to the proof of the
Main Lemma~\ref{main-puiseux-l}.
Here we list some basic preliminary bounds.
Specifically, for future reference,
we record the size of $\partial_z \varphi_\beta(z)$ and
$\partial_\beta \varphi_\beta (z)$.
Then,
in \S~\ref{s:satellite},
our work towards proving \eqref{eq:nidentity} for $n=p$
will start (and finish) at $n=p$, for satellite parameters, and at $n=q \ell_0+1$, for parameters with non-central returns.

The critical points of $\varphi_\beta$ 
are $\upomega=0$ and $\upomega' = -2t/\beta$ (resp. $-2t^{-1}/\beta$) if $\val{\beta}\le 1$ (resp. if $\val{\beta} > 1$).
Assuming that $\beta$ lies in some limb $\cL_\theta$,
denote by $D'_0$ the maximal open ball contained in $B_0$
such that $\upomega' \in D_0'$.
Let
$$B'_0:=B_0 \setminus D_0'.$$
and
$$B_{1/2} := C_1 \cup \cdots \cup C_{q-1},$$
where $C_j$ for $1\le j < q$
are the disks from Corollary~\ref{c:action-balls} (2). 
Note that $B_{1/2}$ is the union of $q-1$ balls that depend on $\beta$,
but we omit this dependence in the notation. If $\beta_0$ is  such that $\upomega$ is periodic, then $\upomega_n(\beta_0) \in B_{1/2}$
whenever $\upomega_n(\beta_0) \in D_0$, for otherwise, $\upomega_{n+1}(\beta_0) \in U_0$.

\begin{lemma}
  \label{l:size-derivatives}
  Let $\cL_\theta$ be a limb of order $q \ge 2$.
  Assume $\beta \in \cL_\theta$.

  If $q=2$, then:
  \begin{eqnarray}
    \label{eq:zder2}
\|\partial_z \varphi_\beta(z)\| & = &
\begin{cases}
  \val{z}\cdot \val{t}      & \text{ if } z \in B'_0,\\
  \val{t}^{-1} & \text{ if } z \in B_1 \cup B_{1/2}.
\end{cases}  \\
    \label{eq:bder2}
  \|\partial_\beta \varphi_\beta (z)\|&=&
\begin{cases}
  \val{z}^3 \cdot \val{t}^2      & \text{ if } z \in  B_0' ,\\
   1        & \text{ if } z \in  B_{1/2},\\
  1/\val{z}      & \text{ if } z \in B_1.
\end{cases}
\end{eqnarray}
If $q \ge 3$, then:
\begin{eqnarray}
  \label{eq:zderq}
\|\partial_z \varphi_\beta(z)\| & = &
\begin{cases}
  \val{z}\cdot \val{t}^{-1}      & \text{ if } z \in B'_0 ,\\
  \val{t}^{-2} & \text{ if } z \in B_{1/2},\\
  1 & \text{ if } z \in B_1 \cup \cdots \cup B_{q-1}.
\end{cases}  \\
  \label{eq:bderq}
  \|\partial_\beta \varphi_\beta (z)\|&=&
\begin{cases}
  \val{z}^3 \cdot \val{t}^{-2}      & \text{ if } z \in  B_0' ,\\
   \val{t}^{-1}        & \text{ if } z \in  B_{1/2},\\
   1/\val{z}      & \text{ if } z \in B_1,\\
   1 & \text{ if } z \in B_2 \cup \cdots \cup B_{q-1}.
\end{cases}
\end{eqnarray}
\end{lemma}

\begin{proof}
  The computations are straightforward and left to the reader.
\end{proof}

Denote by $\diam (B)$ the spherical diameter of a ball contained in $\mathfrak{O}$
or in $\pone_\L \setminus \mathfrak{O}$. 
Observe that if $z \in B_j$ for some $j\neq 0$, then
$$\| \partial_z \varphi_\beta (z) \| = \dfrac{\diam (B_{j+1})}{\diam (B_j)},$$
subscripts mod $q$.
If
$z \in B_0'$ and $\val{z} = \diam (B_0)$,
then 
$$\| \partial_z \varphi_\beta (z) \| = \dfrac{\diam (B_{1})}{\diam (B_0)}.$$
This can be checked directly as an application of Lemma~\ref{l:size-derivatives} or deduced
from Schwarz Lemma. 

The semiconjugacy $\rd(L_\theta(z))$ furnished by Lemma~\ref{l:parabolic-family}
maps the open ball $D_0'$
onto a critical point $\omega'$ which lies in the basin of the multiple
fixed point, 
under the corresponding
parabolic map. Hence,  the orbit of $\upomega$ is disjoint from $D_0'$
and every point in $D_0'$ has infinite orbit.


\subsection{Satellite and first non-central return}
\label{s:satellite}

We now prove the initial cases of \eqref{eq:nidentity} according
to whether a first non-central return exists or not.

\begin{lemma}
  \label{l:satellite}
Consider a parameter  ${\beta_0}$  in an order $q$ limb such that $\upomega$ has period $p$.
Then the following statements hold:
\begin{enumerate}
\item
If $\beta_0$ is a satellite parameter, then
\begin{equation*}
  \|{\partial_\beta \upomega_{p}(\beta_0)}\| = \|{\partial_z \varphi^{\circ p-1}_{\beta_0} (\infty)}\| \cdot \val{\tau}^2.
\end{equation*}

\item
If  $q\ell_0  \ge 1$ is the first non-central return time of $\upomega$ under $\varphi_{\beta_0}$, then
\begin{eqnarray*}
\| \partial_z \varphi^{q \ell_0-1}_{\beta_0}  (\infty) \| &=& \val{\tau}^{-1}\\
 \|{\partial_\beta \upomega_{q \ell_0+1}(\beta_0)}\|
\cdot \val{\tau} &=& \|{\partial_z \varphi^{\circ q\ell_0}_{\beta_0}
(\infty)}\|\cdot \val{\tau}^2.
\end{eqnarray*}
\end{enumerate}
\end{lemma}

\begin{proof}
  If $\beta_0$ is a satellite parameter, it will be convenient to let $\ell_0=p/q$.
  Otherwise, the first non-central return is $q\ell_0$ as is in the statement
  of the Lemma.
  
  We consider first the case in which $\beta_0$ lies in
  the $1/2$-limb. In this limb, for $\val{\gamma}=1$, 
  let $\lambda(\gamma):= \gamma/t$. Then,
$\lambda(\gamma_0) = \beta_0$ for  $\gamma_0 := t \beta_0$.
In view of Lemma~\ref{l:parabolic-family}, for all
$z \in B_0 \setminus D_0$,
$$\rd (\gamma \varphi^{\circ 2}_{\lambda(\gamma)} (z/\gamma)) =
R_\tgamma(\tilde{z}).$$

Hence,
\begin{eqnarray*}
  \rd(\upomega_{2 \ell_0}(\beta_0) + \gamma_0
   \partial_\beta \upomega_{2 \ell_0}(\beta_0) t^{-1}) &=&
  \rd\left(\dfrac{d}{d\gamma}(\gamma \upomega_{2\ell_0}(\lambda(\gamma)))|_{\gamma=\gamma_0} \right) \\
& =& \dfrac{d}{d \tgamma} \rd \left(\gamma \upomega_{2 \ell_0}(\lambda(\gamma) \right)|_{\tgamma=\tgamma_0}\\
                                                                                            &=& \dfrac{d}{d\tgamma} R^{\circ \ell_0}_{\tgamma}(0)|_{\tgamma=\tgamma_0}\\
                                                                                            &= & C \neq 0.
\end{eqnarray*}
The first line is an immediate consequence of the product and chain rules taking into account that $\beta=\gamma/t$. Reduction and derivatives commute 
in $B_0 \setminus D_0$, which yields the second line. Reduction and iteration also
commute in $B_0 \setminus D_0$, thus the third line follows.
Finally, note that $R^{\circ \ell_0}_{\tgamma_0} (0) =0$ or $-1$, so the last line
is a consequence of Lemma~\ref{l:transversality}.

If $2 \ell_0 =p$, then $\upomega_{2 \ell_0} (\beta_0) =0$.
We conclude that $$\| \partial_{\beta} \upomega_{2 \ell_0} (\beta_0) \| =
\val{\partial_{\beta} \upomega_{2 \ell_0} (\beta_0)} = \val{t}.$$

If $2 \ell_0 <p$, then
\begin{eqnarray*}
\upomega_{2 \ell_0} (\beta_0) &=&1/\tgamma_0 + o(1).\\
\partial_\beta \varphi_\beta (\upomega_{2 \ell_0} (\beta_0)) &=& \tgamma_0 + o(1).\\
\partial_z \varphi_{\beta_0}(\upomega_{2 \ell_0} (\beta_0)) & =&
                                                       -\beta_0 \cdot \tgamma_0^2 - 2 \tgamma_0^3/t\\
                                               & = & \tgamma_0^3/t.
\end{eqnarray*}                                                       
Therefore,
\begin{eqnarray*}
 \partial_\beta \upomega_{2 \ell_0 + 1} (\beta_0) &=&  \partial_\beta \varphi_\beta (\upomega_{2 \ell_0}(\beta_0)) + \partial_z \varphi_\beta(\upomega_{2 \ell_0}) \partial_\beta \upomega_{2 \ell_0} ({\beta}_0)\\
                                                   &=& {\tgamma_0} + \dfrac{\tgamma_0^3}{t} \cdot t \dfrac{C-1/\tgamma_0}{\tgamma_0} +o(1)\\
  &=&\tgamma_0^2 C +o(1).  
\end{eqnarray*}
Hence,
$$\|\partial_\beta \upomega_{2 \ell_0 + 1} (\beta_0)\| = \val{\partial_\beta \upomega_{2 \ell_0 + 1} (\beta_0)}=1,$$
since $\upomega_{2 \ell_0 + 1} (\beta_0) \in B_0$.

By~\eqref{eq:zder2},
$$\| \partial_z \varphi^{2 \ell_0-1}_{\beta_0}  (\infty) \| = \val{t}^{-1},$$
and the lemma follows for $q=2$.


Now consider the case in which $q \ge 3$.
For $\gamma \in \mathfrak{O}$, let $\lambda(\gamma) := c_\theta + \lambda t$
be the standard parametrization of $\cL_\theta$.
Then, $\lambda(\gamma_0) = \beta_0$ for $\gamma_0 := (\beta_0-c_\theta)/t$.
Recall that  for all $z \in B_0\setminus D_0 $,
$$\rd (t^{-1} c_\theta \varphi^{\circ q}_{\lambda(\gamma)}
(z/t^{-1} c_\theta)) = R_{v(\tgamma)}(\tilde{z})$$
where, $v(\tgamma)$ is a $\C$-affine automorphism (Lemma~\ref{l:parabolic-family}).
As above, the chain rule and commutative properties of derivatives, reduction and iteration yield:
\begin{eqnarray*}
\rd(c_\theta\partial_\beta \upomega_{q \ell_0}(\beta_0)) &=&
                   \rd\left(\dfrac{d}{d\gamma}( t^{-1}c_\theta\upomega_{q\ell_0}(\beta(\gamma)))|_{\gamma=\gamma_0} \right) \\
& =& \dfrac{d}{d \tgamma} \rd \left(t^{-1} c_\theta \upomega_{q \ell_0}(\beta(\gamma) \right)|_{\tgamma=\tgamma_0}\\
                                                                                            &=& \dfrac{d}{d\tgamma} R^{\circ \ell_0}_{v(\tgamma)}(0)|_{\tgamma=\tgamma_0}\\
                                                                                                &\neq & 0.
\end{eqnarray*}

If $q \ell_0 =p$, then $\upomega_{q \ell_0} (\beta_0) =0$.
We conclude that $$\nor{\partial_\beta \upomega_{q \ell_0} (\beta_0)}=\val{\partial_\beta \upomega_{q \ell_0} (\beta_0)} = 1.$$

If $q \ell_0 <p$, then
\begin{eqnarray*}
\upomega_{q \ell_0} (\beta_0) &=&t/c_\theta + o(t).\\
\partial_\beta \varphi_\beta (\upomega_{q \ell_0} (\beta_0)) &=& c_\theta t^{-1} + o(t^{-1}).\\
\partial_z \varphi_{\beta_0}(\upomega_{q \ell_0} (\beta_0)) & =&
\beta_0 c^2_\theta t^{-2} - 2t c_\theta^3 t^{-3} + o(t^{-2})\\
& = & c_\theta^3 t^{-2} +o(t^{-2}).
\end{eqnarray*}                                                       

Therefore,
\begin{eqnarray*}
  \val{\partial_\beta \upomega_{q \ell_0 + 1} (\beta_0)} &=& \val{\partial_\beta \varphi_\beta (\upomega_{q \ell_0}(\beta_0)) + \partial_z \varphi_{\beta_0}(\upomega_{q \ell_0}(\beta_0)) \partial_\beta \upomega_{q \ell_0} ({\beta}_0)}\\
     &=&\val{\partial_z \varphi_{\beta_0}(\upomega_{q \ell_0}(\beta_0))}.
\end{eqnarray*}
Moreover, $\upomega_{q\ell_0} (\beta_0) ,\upomega_{q\ell_0 + 1} (\beta_0) \notin B_1$ so
$$\nor{\partial_\beta \upomega_{q \ell_0 + 1} (\beta_0)}=\nor{\partial_z \varphi_{\beta_0}(\upomega_{q \ell_0}(\beta_0))}.$$
Furthermore, $$\nor{\partial_z \varphi_{\beta_0}^{q \ell_0-1} (\infty)} = 1,$$
by \eqref{eq:zderq}. For $q \ge 3$, the lemma now follows from the chain rule.
\end{proof}

\subsection{Geometry of polynomial rescaling limits}

Given $\beta$ in an order $q$ limb $\cL_\theta$, the \emph{filled
  Julia set} of $\varphi_\beta$ is
$$\cK(\varphi_\beta) := \{ z \in \ponel : \varphi^{\circ n}_\beta (z) \notin U_0 \text{ for all } n \ge 0 \}.$$
We are interested on the geometry of $\cK(\varphi_\beta)$ around $\upomega$
for parameters with non-central returns.

\begin{proposition}
  \label{p:x}
  Let $\beta_0$ be a  parameter of period $p$ in an order $q$ limb
  such that $\upomega$ has a non-central return.
  Then there exist a periodic rational closed ball $X_0$
  around $\upomega$ of period $q'>q$. Moreover, 
  $\varphi^{\circ q'}_\beta : X_0 \to X_0$ has degree $2$.
\end{proposition}
Clearly, $X_0 \subset \cK(\varphi_{\beta_0})$ and $q'$ divides $p$.

\begin{proof}
  The proposition is consequence of results from \cite{KiwiPuiseuxQuadratic}.
  Let 
  $D$ be the maximal open ball such that $\upomega \in D \subset \cK(\varphi_\beta)$. 
  Let $X_0$ 
  be the closed ball containing $\upomega$ of  diameter $\diam (D)$.
  By~\cite[Proposition~5.3]{KiwiPuiseuxQuadratic}, $X_0$ is 
  a rational closed ball and either $X_0=B_0$ or $X_0$ is periodic, say of period $q'$.
  We now prove by contradiction that $X_0 \neq B_0$ . Suppose
$X_0=B_0$. Then  $D$ is a maximal open ball in $B_0$.
Since $\beta_0$ has a non-central return,
$D$ eventually maps onto $D_0$. This is
impossible because $D_0$ contains a preimage of $U_0$.
Therefore, $X_0$ is a periodic ball. Moreover, $q'>q$, since $X_0$ maps into $D_0$ in $q\ell_0$ iterates, for some $\ell_0 \ge 1$. Furthermore,
the unique critical point of $\varphi^{\circ q'}_\beta : X_0 \to X_0$ is $\upomega$, so it has degree $2$.
\end{proof}

\begin{lemma}
  \label{l:unique-ball}
  Suppose that $\beta$ is a parameter in an order $q$ limb.
  If there exists $q' \ge 1$ and a ball $X_0 \ni \upomega$ such that
  $\varphi^{q'}_\beta (X_0) = X_0$, then either $X_0$ is an open ball
  and $\diam (X_0) = \diam (B_0)$, or $X_0$ is the unique periodic
  closed ball containing $\upomega$, $q' >q$ and
  $\upomega$ has a non-central return.
\end{lemma}

\begin{proof}
  Since $\partial \overline{X_0} \subset \poneberk$ is a (Berkovich) repelling
  periodic point of $\varphi_\beta: \poneberk \to \poneberk$, according to~\cite[Theorem 1]{KiwiPuiseuxQuadratic},
  we have that $\partial \overline{X_0} = \partial \overline{B_0}$ or $X_0$ is a closed rational
  ball properly contained in $B_0$. 
  In the first case, taking into account that some points of $B_0$ map into $U_0$, we have
  that $B_0$ is not a periodic ball so $X_0$ is a maximal open ball
  of $B_0$.
  In the second, we claim that $\upomega$ has a non-central return.
  By contradiction, suppose that $\beta$ is a
  satellite parameter, then $\varphi_\beta^{\circ q'}(D') =D'$, where $D'$
  is the open ball around $\upomega$ of diameter $\diam(B_0)$.
  By Schwarz Lemma, $\varphi^{\circ q'}_\beta(X_0) \subsetneq X_0 \subsetneq D'$ which
  is a contradiction. Therefore, $\upomega$ has a non-central return and $q'>q$. 
\end{proof}


The size of the balls in the orbit of $X_0$ is controlled by the following result.

\begin{lemma}
  \label{l:smallest-ball}
  Let $\beta$ be a parameter in an order $q$ limb.
  Suppose that $X_0$ is a periodic rational closed ball
  of period $q'$ containing $\upomega$ and $q\ell_0$ 
  is the first non-central return time of $\upomega$.
  Given $0 \le j <q'$, consider $k$ such that
  $$X_j := \varphi^{\circ j}_{\beta_0} (X_0) \subset B_k$$
  and let
  $$L_j := \log \dfrac{\diam (B_k)}{\diam (X_j)}.$$
 Then $$L_1 = \dots =L_{q\ell_0} > L_j,$$
  for all $j=q \ell_0 +1, \dots, q'$.
\end{lemma}

Given a  parameter $\beta$ with a periodic ball $X_0$ as above, we will freely use the notation
for 
$X_j$ and $L_j$ introduced in the lemma, subscripts mod $q'$.

\begin{proof}
  Given $X_j$, let $k$ be such that  $X_j \subset B_k$ and
  $$\Gamma_j := \{ B \subset \ponel : B \text{ is a closed ball and } X_j \subset B \subsetneq B_{k} \}.$$ We use
  reverse interval notation in $\Gamma_j$, which is totally ordered by inclusion. That is, $[A,B] \subset \Gamma_j$ consists of all the
  balls $C$ such that $A \supset C \supset B$.
  Moduli of annuli give a natural  parametrization of  $\Gamma_j$ by a subinterval of $\R$
  (corresponding to the hyperbolic distance in $\poneberk \setminus \ponel$):
$$  \begin{array}{ccc}
      \Gamma_j & \to &   ]-\log \diam (B_{k}),-\log \diam (X_j)]\\
      B & \mapsto & -\log \diam(B).
    \end{array}
    $$
    We simply say that the length of $[A,B]$ is
    the modulus of $A \setminus B$, so the parametrization preserves length.

  We omit the subscript $\beta_0$ and simply write $\varphi$.
  We will analyze the dynamics induced by $\varphi$ on $\cup \Gamma_j$.
  Its action in $\Gamma_j$ will be also denoted by $\varphi$.
  Recall that $\varphi$ maps $D_0 \setminus \cup C_j$ bijectively onto $U_0$.
Denote by $D$ the open ball around $\upomega$ of diameter $\diam (B_0)$.

  If $X_j \subset D_0$, then $X_j \subset C_k$ for some $k$ and
   the portion $]B_0,C_k[$ of
  $\Gamma_j$ maps outside of $\cup \Gamma_j$; that is, $\varphi(B) \notin
  \cup \Gamma_j$, for all $B \in ]B_0,C_k[$.
  Moreover, $[C_k,X_j]$ maps bijectively onto $\Gamma_{j+1}=]B_{k+1},X_{j+1}]$.
  Furthermore, $\varphi: [C_k,X_j] \to \Gamma_{j+1}$ is length preserving since
  $C_k$ maps bijectively onto $B_k$. 

  If $X_j \subset B_k$ for some $k \neq 0$, then $\varphi: \Gamma_{j} \to
  \Gamma_{j+1}=]B_{k+1},X_j]$ is a length preserving bijection, since
  $\varphi$ maps $B_k$ onto $B_{k+1}$ bijectively.

  The interval $]B_0,X_0]$  consists of balls contained in $D$ and  $\varphi:
  A \setminus B \to \varphi(A)\setminus \varphi(B)$ is a degree $2$ map
  if $]A,B] \subset ]B_0,X_0]$. Therefore,
  $\varphi:\Gamma_0 \to \Gamma_1$ is a bijection that multiplies lengths by a factor of $2$.  
  
  If $X_j \subset D$ for some $j\neq 0$, then $\Gamma_j \cap \Gamma_0 =]B_0, S_j]$ for
  some $S_j$. Hence, $\varphi: \Gamma_j \to \Gamma_{j+1}$ is a bijection that
  multiplies the length of subintervals of  $]B_0,S_j]$ by a factor $2$ and preserves lengths of the subintervals of $[S_j,X_j]$.
  
  Note that $\varphi(\Gamma_j) \supset \Gamma_{j+1}$ for all $j$ and equality
 holds only if $X_j$ is not contained in $D_0$.


  Now we are ready to compare lengths of $\Gamma_j$ with $\Gamma_1$.
  We claim that, for all $j \ge 1$, there
  exists an injective
  piecewise continuous and
  length preserving map $\psi: \Gamma_j \to \Gamma_1$. Indeed, given $B \in \Gamma_j$, let $n=n(B) \ge 0$ be minimal so that
  there exists $B' \in \Gamma_1$ with the property that
  $\varphi^{\circ n}(B') =B$.
  Such a number $n(B)$ always exists. In fact, $B$ has a unique preimage  $C$ in $\Gamma_{j-1}$. If $C \in \Gamma_0$, then $B\in \Gamma_1$ and $n(B)=0$.
  Otherwise, $C$ maps injectively onto $B$.
  Taking preimages recursively we obtain $n(B)$ and $B' \in \Gamma_1$.
Let $\psi: \Gamma_j \to \Gamma_1$ be the map that sends $B$ to $B'$.
  Note that $n(B) < j$. Moreover, $n(B)$ is non-decreasing. 
  Given $n \ge 0$, let
  $$I_n := \{ B \in \Gamma_j : n = n(B) \}.$$
  If $B'=\psi(B)$ for some  $B \in I_n$,
  then  $B', \varphi(B'), \dots, \varphi^{n-1}(B') \notin \Gamma_0$, for otherwise, $n(B) < n$. Hence, 
 $\psi: I_n \to \Gamma_1$
  preserves length.
  Moreover, $L_1 \le L_j$, since from the definition
  of $n(B)$ we have that  $\psi(I_n) \cap \psi (I_m) = \emptyset$ if $n \neq m$.
  
  If $n(B)$ is constant, say $n$, in $\Gamma_j$, then $\varphi^{\circ n}: \Gamma_1 \to \Gamma_j$ is a bijection. This can only occur if $L_1=L_j$ and $j \le q \ell_0$.

  If $j > q\ell_0$, then $n(B)$ is not constant. Consider $n' > n$
  such that the intervals $I:=I_n$, $I':=I_{n'}$ are not empty.
  Let $J = \psi(I)$ and $J' = \psi(I')$. We claim that the endpoint $B$ of $J$
  cannot be initial point of $J'$. Otherwise, $\varphi^{\circ n}(B)$ would
  be periodic of period diving $n'-n$, contradicting
  Lemma~\ref{l:unique-ball}. It follows that $\psi(L_j)$
  omits intervals of $L_1$ and, therefore, $L_j < L_1$.
\end{proof}

The dynamical space derivate at $\infty$
is controlled by the diameters of $X_j$:

\begin{lemma}
  \label{c:zder}
    Assume that $\beta_0$ is a parameter in an order $q$ limb such
    that $\upomega$ is periodic of  period $p$ under $\varphi_{\beta_0}$ and
    $\{X_0,\dots, X_{q'-1}\}$ is its cycle of periodic closed balls such that $\upomega \in X_0$. Then, for all $j \le p$:
    $$\nor{\partial_z \varphi_{\beta_0}^{\circ j-1}(\infty)} = \dfrac{\diam(X_j)}{\diam(X_1)}.$$   
  \end{lemma}

  \begin{proof}
    This is an application of the Schwarz Lemma (SL).
    Indeed, since $\varphi_{\beta_0}^{\circ j-1}$ maps $X_1$ onto $X_j$, it maps
    the open ball $D_1$ of diameter $\diam(X_1)$ containing $\infty$ onto
    the open ball $D_j$ of diameter $\diam(X_j)$ containing $\upomega_{j}(\beta_0)$.
    Every element of $D_1$ is in the basin of
    attraction of the periodic critical orbit (SL). Therefore, $D_1, \dots, D_{p-1}$
    cannot contain a critical point. Hence, $\varphi_{\beta_0}^{\circ j-1}:D_1\to D_j$ is a bijection for all $j \le p$ and the lemma follows from (SL).
\end{proof}

\subsection{Polynomial parameter rescaling}
\label{s:primitive}
Consider a parameter $\beta_0$ such that $\upomega$ has period $p$ and a non-central return. Recall that $X_0, \dots, X_{q'-1}$ denotes the periodic  closed balls under $\varphi_{\beta_0}$ where
$\upomega \in X_0$.
Our analysis will show that there exists a parameter space ball $\Lambda_0$,
around $\beta_0$,  such that $\{X_0, \dots, X_{q'-1}\}$
 is a periodic orbit of closed balls
 \emph{for all} $\beta \in \Lambda_0$.
We say that parameters $\beta$ with such an orbit of periodic closed balls,
have  \emph{a period $q'$ primitive renormalization}.

Before proving the existence of $\Lambda_0$,
we need to prove that \eqref{eq:nidentity} holds for $n=q'$:

\begin{lemma}
  \label{l:polynomial}
  Assume that $\beta_0$ is a parameter in an order $q$ limb such
  that $\upomega$ is periodic of  period $p$ under $\varphi_{\beta_0}$.
  Suppose that $\varphi_{\beta_0}$ has a primitive
  renormalization of period $q'$.
Let $\ell_0 \ge 1$ be such that $\upomega_{q \ell_0}(\beta_0)$ is the first non-central return of $\upomega$.
Then, for all $q \ell_0 + k \le q'$:
$$\|\partial_\beta \upomega_{q \ell_0 + k} ({\beta_0})\| =
\| \partial_z \varphi^{\circ k}_{\beta_0} (\upomega_{q\ell_0}) \| \cdot \val{\tau}.$$
Moreover,
  $$\| \partial_\beta \upomega_{q'} ({\beta_0})\| = \|\partial_z \varphi^{\circ q'-1}_{\beta_0}(\infty)\| \cdot \val{\tau}^2.$$ 
\end{lemma}
\begin{proof}
  From Lemma~\ref{l:satellite}, recall that $ \|\partial_z \varphi_\beta^{\circ q \ell_0-1} (\infty)\| = \val{\tau}^{-1}$. Thus, the second formula follows from
  the first.
  To prove the first formula we proceed by induction. 
  In view of Lemma~\ref{l:satellite}, the case $k=1$ is already proven.

  For $ q \ge 3$, all the balls $B_0, \dots, B_{q-1}$ have
  the same diameter.  Hence, from Lemma~\ref{l:smallest-ball}, we conclude that
  $\diam(X_{q\ell_0}) < \diam(X_{q\ell_0 +k}).$
  For $q=2$, the same holds unless $X_{q\ell_0+k} \subset B_1$; in this case,
  $\val{t} \cdot
  \diam(X_{q\ell_0}) < \diam(X_{q\ell_0 +k}).$
For all $k\le q'-q\ell_0$, since $\varphi^{\circ k}_{\beta_0}: X_{q\ell_0} \to
  X_{q \ell_0 +k}$ is a bijection,
  Schwarz Lemma yields:
\begin{equation}
  \label{eq:expansion}
  \| \partial_z \varphi_{\beta_0}^{\circ k} (\upomega_{q\ell_0}(\beta_0)) \| >
  \begin{cases} \val{t},  \text{ if $q=2$ and $X_{q\ell_0+k} \subset B_1$,}\\
    1, \text{ otherwise.}
  \end{cases}
\end{equation}
Let $z =: \upomega_{q \ell_0 +k} (\beta_0)$.
If $q \ge 3$, or $q=2$ and $X_{q\ell_0+k} \subset B_0$, then
\begin{eqnarray*}
  \|\partial_{\beta} \varphi_{\beta} (z) \|& \le&  \| \partial_z
\varphi_{\beta_0} (z) \| \cdot \val{\tau}\\
&<& \| \partial_z \varphi_{\beta_0} (z) \|  \cdot  \| \partial_z \varphi_{\beta_0}^{\circ k} (\upomega_{q\ell_0})\| \cdot
    \val{\tau}
\end{eqnarray*}    
The first inequality can be checked case by case applying Lemma~\ref{l:size-derivatives}. The second is a consequence of \eqref{eq:expansion}.

Similarly, 
if $q=2$ and $X_{q\ell_0+k} \subset B_1$, then
\begin{eqnarray*}
  \|\partial_{\beta} \varphi_{\beta} (z) \|& \le&  \| \partial_z
\varphi_{\beta_0} (z) \| \cdot \val{t}^2 \\
&<& \| \partial_z \varphi_{\beta_0} (z) \|  \cdot  \| \partial_z \varphi_{\beta_0}^{\circ k} (\upomega_{q\ell_0})\| \cdot
    \val{\tau}
\end{eqnarray*}

In both cases,  the inductive hypothesis yields 
\begin{eqnarray*}
  \|\partial_{\beta} \varphi_{\beta} (z) \|
  &<& \| \partial_z \varphi_{\beta_0} (z) \|  \cdot \|\partial_\beta \upomega_{q \ell_0 + k} ({\beta_0})\|.
  \end{eqnarray*}

  By the chain rule, the strong triangle inequality, and the inductive hypothesis
  we have:
\begin{eqnarray*}
  \|\partial_\beta \upomega_{q \ell_0 +k+ 1} (\beta_0)\| &=& \|\partial_\beta \varphi_\beta (z) + \partial_z \varphi_\beta(z) \partial_\beta \upomega_{q \ell_0+k} ({\beta}_0)\| \\
                                                           &=& \| \partial_z \varphi_\beta(z) \partial_\beta \upomega_{q \ell_0+k} ({\beta}_0)\|\\
   &=& \|\partial_z \varphi_{\beta_0}^{\circ k+1} (\upomega_{q\ell_0}(\beta_0))\| \cdot \val{\tau} 
  \end{eqnarray*}

\end{proof}

\begin{proposition}
  \label{p:parameter-poly}
  Let $\beta_0$ be such that $\upomega$ has period $p$ under $\varphi_{\beta_0}$.
  Suppose that 
  $\{X_0, \dots, X_{q'-1}\}$ is a periodic cycle of
  (rational) closed balls where $\upomega \in X_0$. Consider
  $$\Lambda_0 := \{ \beta \in \L : \val{\beta - \beta_0} \le \diam({X}_1) \cdot |\tau|^{-2}\}.$$
  Then all of the following statements hold:
  \begin{enumerate}
  \item For all $\beta \in \Lambda_0$, we have $\varphi_\beta (X_j)=X_{j+1}$, subscripts modulo $q'$.
  \item The map $\upomega_{q'}: \Lambda_0 \to X_0$ is 
    a bijection.
  \item For $j =0,1$, let $\rho_j\in \QS$ be such that
    $\diam (X_j) =\val{t}^{\rho_j}$.
    Consider:
    \begin{eqnarray*}
      M_0(z) &:=& t^{\rho_0} z,\\
      \mu_0 (\gamma) &:=& \beta_0 + \tau^{-2} t^{\rho_1} \gamma \text{ for } \gamma \in \mathfrak{O}.
    \end{eqnarray*}
    Then,   there exist  $A,B,C \in \C$
  with $A,B \neq 0$, such that
    $$\rd( M_0^{-1} \circ  \varphi^{\circ q'}_{\mu_0 (\gamma)} \circ  M_0(z)) = A \tilde{z}^2 + B \tgamma + C,$$
  for all $z \in X_0$ and $\gamma \in \mathfrak{O}$. Moreover, $\omega=0$ has period $p/q'$ under
  $\tilde{z} \mapsto A \tilde{z}^2 + C$.

  %
  \end{enumerate}
\end{proposition}

\begin{remark}
  In Berkovich space language, Proposition~\ref{p:parameter-poly} (3) says
  that the following map is an affine isomorphism:
  $$\begin{array}{ccc}
     \C& \to & \cS_1\\
     \tgamma& \mapsto & T_{\zeta_0} \varphi^{q'}_{\mu_0(\gamma)}
  \end{array}$$
  where $\zeta_0 \in \poneberk$
  is the $\sup$-norm in $X_0$; equivalently $\{\zeta_0\} = \partial \overline{X}_0$.
\end{remark}



Before proving Proposition~\ref{p:parameter-poly} we need the following:

\begin{lemma}
    Suppose that 
  $\{X_0, \dots, X_{q'-1}\}$ is a periodic cycle of
  (rational) closed balls under iterations
  of $\varphi_{\beta'}$, for some $\beta'$ in a limb with $\upomega \in X_0$.
  Consider
  $$\Lambda_0 := \{ \beta \in \L : \val{\beta  - \beta'} \le \diam({X}_1) \cdot |\tau|^{-2}\}.$$
  Then, for all $\beta \in \Lambda_0$, we have $\varphi_{\beta} (X_j)\subset X_{j+1}$ for all $j$.
\end{lemma}
\begin{proof}
   Consider $\beta \in \Lambda_0$. Our proof is a case by case
  analysis according to the position of $X_j$.
From Lemma~\ref{l:smallest-ball}, $$\diam({X}_j) \ge
  \begin{cases}
    \diam({X}_1) \cdot |\tau|^{-1}, & \text{ if } X_j \subset B_0, \\
    \diam({X}_1), & \text{ otherwise. }
  \end{cases}
  $$
 
  Suppose $X_j \subset B_0' \setminus D_0$. Then $X_{j+1} \subset B_1$.
  By \eqref{eq:bder2} and \eqref{eq:bderq},  $\|\partial_\beta \varphi_\beta (z)\|  \le |\tau|^2$ for all $z \in B'_0$.
  Thus, by the Mean Value Theorem, $$\dist_s(\varphi_{\beta_0}(z),\varphi_{\beta} (z)) \le \diam(X_1) \le \diam(X_{j+1}),$$
  for all $\beta \in \Lambda_0$ and $z \in X_j$. That is, $\varphi_\beta(X_j)
  \subset X_{j+1}$.



  Suppose $X_j \subset B_1$. Hence, $X_{j+1} \subset B_0$ if $q=2$ and
  $X_{j+1} \subset B_2$, otherwise. Moreover, by \eqref{eq:bder2} and \eqref{eq:bderq}, 
  if $z \in X_j$ and  $\beta \in \Lambda_0$, then
  $\|\partial_\beta \varphi_\beta (z) \| \le \val{t}$.  Therefore,
  $$\dist_s(\varphi_\beta(z),\varphi_{\beta_0} (z))\le \diam(X_1) \val{\tau}^{-2} \val{t} \le \diam(X_1) \val{\tau}^{-1} \le \diam(X_{j+1}).$$
  That is, $\varphi_\beta(X_j)
  \subset X_{j+1}$.

  Suppose $q \ge 3$ and $X_j \subset B_k$ for some $2 \le k <q$.
  If $z \in X_j$ and  $\beta \in \Lambda_0$, then
  $\|\partial_\beta \varphi_\beta (z) \| = 1$ and
  $$\dist_s(\varphi_\beta(z),\varphi_{\beta_0} (z))\le \diam(X_1) \le \diam(X_{j+1}).$$
  Thus, $\varphi_\beta (X_j) \subset X_{j+1}$.
  
  Finally, suppose that  $X_j \subset B_{1/2}$.
  By \eqref{eq:zder2}-\eqref{eq:bderq}, 
   $\|\partial_\beta \varphi_\beta (z) \| = 1$ and $\|\partial_z \varphi_{\beta_0} (z)\| =\val{t}^{-2} \cdot \val{\tau}$,
   for all $z \in X_j$ and all $\beta \in \Lambda_0$.
   Thus,
   $$\dist_s(\varphi_\beta(z),\varphi_{\beta_0} (z))\le \diam(X_1) \val{\tau}^{-2}  \le \diam(X_j) \val{\tau}^{-1},$$
   and
   $$\diam(X_j) \val{t}^{-2} \cdot \val{\tau} =
   \diam(X_{j+1}).$$
   Therefore, $\dist_s(\varphi_\beta(z),\varphi_{\beta_0} (z)) \le \diam(X_{j+1})$ and $\varphi_{\beta}(X_j) \subset X_{j+1}$.
 \end{proof}

 \begin{proof}[Proof of Proposition~\ref{p:parameter-poly}]
   Given $\beta \in \Lambda_0$, by the previous lemma, we have that
   $\varphi_\beta(X_j) \subset X_{j+1}$ for all $j$.
   It follows that $\upomega$ has a non-central return under $\varphi_\beta$.
 To prove that $\varphi_\beta(X_j) = X_{j+1}$, we proceed by contradiction.
   If $\varphi_\beta(X_j) \subsetneq X_{j+1}$ for some $j$, then
   there exist a closed ball $X_0' \supsetneq X_0$ which is periodic
   under $\varphi_\beta$, maybe of smaller
   period. 
   Let $X_j' := \varphi^{\circ j}_\beta(X'_0)$ and note
   that $\diam(X_1') > \diam(X_1)$. Therefore,
   $$\Lambda'_0 := \{ \beta' \in \L : \val{\beta' - \beta} \le \diam({X}'_1) \cdot |\tau|^{-2}\} \supset \Lambda_0.$$
   Again by the previous lemma, for all $\beta' \in \Lambda_0'$, we
   have that $\varphi_{\beta'}(X_j') \subset X'_{j+1}$. Since
   $\beta_0 \in \Lambda_0'$, it follows
   that $\varphi^{\circ q'}_{\beta_0}(X_0') \subset X_0'$. But $X_0$ is properly
   contained in $X_0'$. By Schwarz Lemma, $X_0$ is in the basin of $\upomega$ under $\varphi_{\beta_0}$, which contradicts the periodicity of $X_0$. Hence,
   (1) holds.
   
   For $0\le j < q'$, we have proven that  $\omega_{q'} : \Lambda_0 \to X_{0}$ is well defined.
   Observe that $\varphi^{\circ q'-1}_{\beta_0}: X_1 \to X_{q'}$ is a bijection. By Schwarz Lemma,
   $$\| \partial_z \varphi_{\beta_0}^{\circ q'-1}(\infty) \| = \dfrac{\diam(X_{q'})}{\diam(X_0)}.$$
Thus, Lemma~\ref{l:polynomial} yields:
$$\| {\partial_\beta \upomega_{q'}} (\beta_0)\| =
{\diam(X_{q'})}{\diam(X_1) |\tau|^{-2}}.$$
By Schwarz Lemma, again,
$\omega_{q'} : \Lambda_0 \to X_{0}$ is a bijection. That is (2) holds.

Taking into account that $\mu_0 (\gamma)$ is
a parametrization of $\Lambda_0$ by $\mathfrak{O}$, that
the map $\varphi_\beta^{\circ q'} : X_0 \to X_0$ has degree $2$, and
that $M_0 (\mathfrak{O})) = X_0$, for all $\tgamma \in \C$:
$$Q_\tgamma := \rd( M_0^{-1} \circ  \varphi^{\circ q'}_{\mu_0 (\gamma)} \circ  M_0)$$
is a quadratic polynomial. Moreover, its critical point is at $z=0$
and  its coefficients are polynomials in $\tgamma$. That is,
$$Q_\tgamma (z) = A (\tgamma) z^2 + V (\tgamma),$$
for some $A(\tgamma), V(\tgamma) \in \C[\tgamma]$.
Since $A(\tgamma)$ has no roots, it must be constant, say $A$.
The critical value $V(\tgamma)$ is $ \rd \omega_{q'} (\mu_0(\gamma))$,
which has degree $1$ by the already proven assertion (2).
That is, the first part of assertion (3) holds. If $\omega=0$ has period
$p'$ under $Q_\tgamma$. Then the open ball $D'$ of diameter
$\diam(X_0)$ containing $\upomega$ has period $q'p'|p$. By Schwarz Lemma
every element $D'$ is attracted to $\upomega$, hence $q'p'=p$, which finises
the prove of (3).
\end{proof}

  From the viewpoint of complex dynamics,
  Proposition~\ref{p:parameter-poly} proves the presence of a Mandelbrot torus
around $\cE^*_x$ in $\cM^{cm}_2$. This Mandelbrot torus is analogous to the ones obtained
by Branner-Hubbard~\cite{BrannerHubbardCubicI}, in the parameter space of cubic polynomials, close to infinity.

Let us be more concrete and consider a puncture $x$ of $\cS_p$
of multiplicity $\mu$ with associated Puiseux series $\beta_0$,
and $\diam(X_i)=\val{t}^{\rho_i}$ for $i=0,1$, as in the statement
of the proposition.  Consider the associated
parametrization $(s^\mu, b(s))$ of $\cE_x$.
Since $L_1=2 \cdot L_0$,  there exists $m_0 \ge 1$ such
that $\rho_0=n_0/m_0$ and $\rho_1=n_1/m_0$ for some $n_0,n_1\ge 1$.
Let
\begin{eqnarray*}
  a(s)&:=&s^{\mu m_0},\\
  b_c(s)&:=& b(s^{m_0})+s^{\mu n_1} c, \quad c \in \C\\
  M_s(z)&:=& \dfrac{s^{\mu n_0}}{A} z.
\end{eqnarray*}
From Proposition~\ref{p:parameter-poly} (3) it follows that, uniformly in compact subsets of $\C$, as $s \to 0$:
$$g_{s,c}(z):=M_s^{-1} \circ f_{a(s),b_c(s)} \circ M_s (z) \to z^2 + AB c + C.$$
We extend the family $g_{s,c}$ to $s=0$ by declaring $g_{0,c}(z):= z^2 + AB c + C.$
Hence, there exists $R > 0$ and $\varepsilon >0$
such that, if $|s| < \varepsilon$
and  $c \in \Lambda:= \{ c \in \C : |AB c + C|\le 4 \}$,
we have that $g_{s,c}: U_{s,c}' \to \D_R$
is a quadratic like map, where $\D_R:=\{ z \in \C : |z| <R\}$
and $U_{s,c}'$ is the preimage of $\D_R$. Moreover,
$\cG_{s_0}:=\{g_{s_0,c}: c \in \Lambda\}$ is a full and unfolded quadratic like family,
for all $s_0$, in the sense of Lyubich~\cite[\S 4.11]{LyubichHairiness}. The connectedness locus $\cM_s$ of $\cG_s$ moves holomorphically
with $s$, for $|s|<\varepsilon$, and $\cM_0$ is an affine copy of the
Mandelbrot set.  In other words, the straightening of $\cM_s$ is realized
as the limit (as $s \to 0$) of a holomorphic motion \emph{inside} of the moduli space of quadratic rational maps. 

Thus, we have established the following:

\begin{corollary}
  \label{c:mandelbrot}
  Let $x$ be a puncture with associated Puiseux series $\beta_0$ that has a period $q'$ primitive renormalization. For all $r>0$ sufficiently large,
  there exists a (punctured) neighborhood $\cE_x^*$ of $x$ and 
  a
  holomorphic family of critically marked quadratic rational maps $(F_\lambda,\omega)$ 
  parametrized by $\cE^*_x \times \D_r$ such that all of the following hold:

  \begin{enumerate}
  \item $\{ [(F_\lambda,\omega)] \in \cM^{cm}_2 : \lambda \in \cE^*_x \times \D_r\}$ is a neighborhood of $\cE_x^*$ and $\cE_x^* = \{ [(F_\lambda,\omega)] \in
    \cM^{cm}_2 : \lambda \in \cE^*_x \times \{c_0\}\}$ for some $c_0$ such that
    $z=0$ has period $p/q'$ under $Q_{c_0}$.
  \item There exist continuously varying smooth Jordan domains $U'_\lambda, U_\lambda$ such that
    $F^{\circ q'}_\lambda : U'_\lambda\to U_\lambda$ is a quadratic like
    map, for all $\lambda \in \cE^*_x \times \D_r$.
  \item The family
    of quadratic like maps
    $(g_\lambda):=(F^{\circ q'}_\lambda : U'_\lambda\to U_\lambda)$ extends analytically
    to $\cE_x \times \D_r$ by the quadratic family. Specifically,
    for all $(x,c) \in \{x\} \times \D_r$,  we may defined
    $g_{(x,c)}$ as a quadratic like
    restriction of $Q_c$ and obtain an analytic family parametrized by $\cE_x \times \D_R$.
  \end{enumerate}  
\end{corollary}

\subsection{Proof of (Main) Lemma~\ref{main-puiseux-l}}
\label{s:primitive-r}


    Assume that $\beta_0$ is a parameter in an order $q$ limb such
    that $\upomega$ is periodic of  period $p$ under $\varphi_{\beta_0}$.
    We must show that:
    $$\|\partial_\beta \upomega_{p}(\beta_0) \| = \val{\tau}^2 \cdot
    \nor{\partial_z \varphi^{\circ p-1}_{\beta_0} (\infty)}.$$
    This was proved, when  $\beta_0$ is a satellite parameter, in Lemma~\ref{l:satellite}.
    Thus, we may assume that $\varphi_{\beta_0}$ has a
    periodic ball $X_0$ containing $\upomega$ of period $q'>q$
    dividing $p$ (Proposition~\ref{p:x}).
In this case, 
    the lemma is ultimately a consequence of the simplicity of the roots of Gleason polynomials. 
    Consider $M_0$, $\mu_0$, $A,B,C$ as in
    Proposition~\ref{p:parameter-poly}.
    Let $$Q_\tgamma (z) := A z + B\tgamma + C.$$
  Then,
  \begin{eqnarray*}
    \rd(\partial_{\gamma} (M_0^{-1} \circ \upomega_p(\mu_o(\gamma)))|_{\gamma=0}) &=&
                                                                     \dfrac{dQ^{\circ p/q'}_{\tgamma}}{d\tgamma}(0)\\
    & \neq & 0.
  \end{eqnarray*}
  Indeed, the first line is obtained changing the order of
  reduction and derivation. The second line is the aforementioned result by Gleason~(Lemma \ref{l:transversality}).
By the chain rule:
$$1 = \dfrac{1}{\diam(X_0)} \nor{\partial_\beta \upomega_p(\beta_0)} \diam(X_1) \val{\tau}^{-2}.$$
From Lemma~\ref{c:zder},
$$\val{\tau}^2 \cdot \nor{\partial_z \varphi^{\circ p-1}_{\beta_0} (\infty)} = \nor{\partial_\beta \upomega_p(\beta_0) }.$$

\qed

\bibliographystyle{alpha}

\begin{thebibliography}{DMWY15}

\bibitem[AK23]{arfeux2023irreducibility}
Matthieu Arfeux and Jan Kiwi.
\newblock Irreducibility of periodic curves in cubic polynomial moduli space.
\newblock {\em Proc. Lond. Math. Soc. (3)}, 127(3):792--835, 2023.

\bibitem[BEK22]{buff2022prefixed}
Xavier Buff, Adam~L. Epstein, and Sarah Koch.
\newblock Prefixed curves in moduli space.
\newblock {\em Amer. J. Math.}, 144(6):1485--1509, 2022.

\bibitem[Ben19]{BenedettoBook}
Robert~L. Benedetto.
\newblock {\em Dynamics in one non-archimedean variable}, volume 198 of {\em
  Graduate Studies in Mathematics}.
\newblock American Mathematical Society, Providence, RI, 2019.

\bibitem[BH88]{BrannerHubbardCubicI}
Bodil Branner and John~H. Hubbard.
\newblock The iteration of cubic polynomials. {P}art {I}: {T}he global topology
  of parameter space, the.
\newblock {\em Acta Math.}, 160(3-4):143--206, 1988.

\bibitem[BKM10]{AKMCubic}
Araceli Bonifant, Jan Kiwi, and John Milnor.
\newblock Cubic polynomial maps with periodic critical orbit. {II}. {E}scape
  regions.
\newblock {\em Conform. Geom. Dyn.}, 14:68--112, 2010.

\bibitem[BR10]{BakerRumelyBook}
Matthew Baker and Robert Rumely.
\newblock {\em Potential theory and dynamics on the {B}erkovich projective
  line}, volume 159 of {\em Mathematical Surveys and Monographs}.
\newblock American Mathematical Society, Providence, RI, 2010.

\bibitem[CA00]{LibroCasas}
Eduardo Casas-Alvero.
\newblock {\em Singularities of Plane Curves}, volume 276 of {\em London
  Mathematical Society Lecture Note Series}.
\newblock Cambridge University Press, Cambridge, 2000.

\bibitem[Cas86]{LibroCassels}
J.~W.~S. Cassels.
\newblock {\em Local Fields}, volume~3 of {\em London Mathematical Society
  Student Texts}.
\newblock Cambridge University Press, Cambridge, 1986.

\bibitem[DeM07]{DeMarcoQuadratic}
Laura DeMarco.
\newblock The moduli space of quadratic rational maps.
\newblock {\em J. Amer. Math. Soc.}, 20(2):321--355 (electronic), 2007.

\bibitem[DH85]{OrsayNotes}
A.~Douady and J.~H. Hubbard.
\newblock {\em \'{E}tude dynamique des polyn\^omes complexes. {I,II}},
  volume~85 of {\em Publications Math\'ematiques d'Orsay [Mathematical
  Publications of Orsay]}.
\newblock Universit\'e de Paris-Sud, D\'epartement de Math\'ematiques, Orsay,
  1985.
\newblock With the collaboration of P. Lavaurs, Tan Lei and P. Sentenac.

\bibitem[DMWY15]{demarco2015bifurcation}
Laura De~Marco, Xiaoguang Wang, and Hexi Ye.
\newblock Bifurcation measures and quadratic rational maps.
\newblock {\em Proc. Lond. Math. Soc. (3)}, 111(1):149--180, 2015.

\bibitem[DS10]{DeMarcoSchiff}
Laura DeMarco and Aaron Schiff.
\newblock Enumerating the basins of infinity of cubic polynomials.
\newblock {\em J. Difference Equ. Appl.}, 16(5-6):451--461, 2010.

\bibitem[Eps]{EpsteinTransversality}
Adam~L. Epstein.
\newblock Transversality in holomorphic dynamics,
  http://homepages.warwick.ac.uk/~mases/Transversality.pdf.

\bibitem[Eps00]{EpsteinBounded}
Adam~Lawrence Epstein.
\newblock Bounded hyperbolic components of quadratic rational maps.
\newblock {\em Ergodic Theory Dynam. Systems}, 20(3):727--748, 2000.

\bibitem[Fis01]{FischerPlaneCurves}
Gerd Fischer.
\newblock {\em Plane algebraic curves}, volume~15 of {\em Student Mathematical
  Library}.
\newblock American Mathematical Society, Providence, RI, 2001.
\newblock Translated from the 1994 German original by Leslie Kay.

\bibitem[FKS23]{FirsovaKahnSelinger}
Tanya Firsova, Jeremy Kahn, and Nikita Selinger.
\newblock On {D}eformation {S}paces of {Q}uadratic {R}ational {F}unctions.
\newblock {\em Int. Math. Res. Not. IMRN}, (8):6703--6738, 2023.

\bibitem[HK17]{HironakaKoch}
Eriko Hironaka and Sarah Koch.
\newblock A disconnected deformation space of rational maps.
\newblock {\em J. Mod. Dyn.}, 11:409--423, 2017.

\bibitem[HT04]{HaTanPinching}
Peter Ha\"issinsky and Lei Tan.
\newblock Convergence of pinching deformations and matings of geometrically
  finite polynomials.
\newblock {\em Fund. Math.}, 181(2):143--188, 2004.

\bibitem[Kiw14]{KiwiPuiseuxQuadratic}
Jan Kiwi.
\newblock Puiseux series dynamics of quadratic rational maps.
\newblock {\em Israel J. Math.}, 201(2):631--700, 2014.

\bibitem[Kiw15]{KiwiRescaling}
Jan Kiwi.
\newblock Rescaling limits of complex rational maps.
\newblock {\em Duke Math. J.}, 164(7):1437--1470, 2015.

\bibitem[KR13]{KiwiRees}
Jan Kiwi and Mary Rees.
\newblock Counting hyperbolic components.
\newblock {\em J. Lond. Math. Soc. (2)}, 88(3):669--698, 2013.

\bibitem[LSvS19]{LWVcor}
Genadi Levin, Weixiao Shen, and Sebastian van Strien.
\newblock Transversality for critical relations of families of rational maps:
  an elementary proof.
\newblock In {\em New trends in one-dimensional dynamics}, volume 285 of {\em
  Springer Proc. Math. Stat.}, pages 201--220. Springer, Cham, [2019]
  \copyright 2019.

\bibitem[Luo22]{LuoTree}
Yusheng Luo.
\newblock Trees, length spectra for rational maps via barycentric extensions,
  and {B}erkovich spaces.
\newblock {\em Duke Math. J.}, 171(14):2943--3001, 2022.

\bibitem[Lyu99]{LyubichHairiness}
Mikhail Lyubich.
\newblock Feigenbaum-{C}oullet-{T}resser universality and {M}ilnor's hairiness
  conjecture.
\newblock {\em Ann. of Math. (2)}, 149(2):319--420, 1999.

\bibitem[Mil93]{MilnorQuadratic}
John Milnor.
\newblock Geometry and dynamics of quadratic rational maps.
\newblock {\em Experiment. Math.}, 2(1):37--83, 1993.
\newblock With an appendix by the author and Lei Tan.

\bibitem[Mil00]{MilnorPOP}
John Milnor.
\newblock Periodic orbits, externals rays and the mandelbrot set: an expository
  account.
\newblock {\em Ast\'erisque}, (261):xiii, 277--333, 2000.
\newblock G{\'e}om{\'e}trie complexe et syst{\`e}mes dynamiques (Orsay, 1995).

\bibitem[NP22]{NiePilgrimBounded}
Hongming Nie and Kevin~M. Pilgrim.
\newblock Bounded hyperbolic components of bicritical rational maps.
\newblock {\em J. Mod. Dyn.}, 18:533--553, 2022.

\bibitem[Ram24]{RamadasGleason}
Rohini Ramadas.
\newblock Moduli spaces of quadratic maps: arithmetic and geometry.
\newblock {\em Int. Math. Res. Not. IMRN}, (15):11364--11370, 2024.

\bibitem[Ree90]{rees1990components}
M.~Rees.
\newblock Components of degree two hyperbolic rational maps.
\newblock {\em Invent. Math.}, 100(2):357--382, 1990.

\bibitem[Ree95]{ReesII}
M.~Rees.
\newblock A partial description of the parameter space of rational maps of
  degree two. ii.
\newblock {\em Proc. London Math. Soc. (3)}, 70(3):644--690, 1995.

\bibitem[Ree03]{ReesAsterisque}
Mary Rees.
\newblock Views of parameter space: {T}opographer and {R}esident.
\newblock {\em Ast\'erisque}, (288):vi+418, 2003.

\bibitem[RS24]{RamadasSilversmithEquations}
Rohini Ramadas and Rob Silversmith.
\newblock Equations at infinity for critical-orbit-relation families of
  rational maps.
\newblock {\em Exp. Math.}, 33(3):379--399, 2024.

\bibitem[Sil98]{SilvermanModuli}
Joseph~H. Silverman.
\newblock The space of rational maps on {$\bold P^1$}.
\newblock {\em Duke Math. J.}, 94(1):41--77, 1998.

\bibitem[Sti93]{Stimson}
James Stimson.
\newblock {\em Degree two rational maps with a periodic critical point}.
\newblock PhD thesis, University of Liverpool, 1993.

\end{thebibliography}

\end{document}